\documentclass{amsart}
\usepackage{amsmath}
\usepackage{amsthm}
\usepackage{amssymb}
\usepackage{color}
\usepackage{graphicx}
\usepackage{apptools}
\usepackage{chngcntr}
\usepackage{amsfonts}
\usepackage{tikz}
\usepackage{makecell}
\usepackage{hyperref}
\usepackage{enumitem}
\usepackage{float}
\usepackage{diagbox}
\usepackage{subcaption}
\usepackage{ifthen}
\usepackage{xcolor}
\usepackage{placeins}

\newtheorem{thm}{Theorem}

\newtheorem{lemma}[thm]{Lemma}

\newtheorem{defi}[thm]{Definition}
\newtheorem{prop}[thm]{Proposition}
\newtheorem{rk}[thm]{Remark}

\newtheorem{conj}[thm]{Conjecture}

\usetikzlibrary{arrows.meta, positioning}

\DeclareMathOperator{\diag}{Diag}
\DeclareMathOperator{\tr}{Tr}
\DeclareMathOperator{\var}{Var}
\DeclareMathOperator{\bernoulli}{Bernoulli}
\DeclareMathOperator{\poisson}{Poisson}
\DeclareMathOperator{\binomial}{Binomial}
\DeclareMathOperator{\nnz}{nnz}
\DeclareMathOperator{\rank}{rank}

\newcommand{\toninfty}{\;\xrightarrow[n\to\infty]{}\;}
\newcommand{\eps}{\varepsilon}

\begin{document}
\title[Limiting spectral distribution of the Watts--Strogatz random graph]{Limiting spectral distribution for the adjacency matrix of the Watts--Strogatz random graph}

\author{Gr\'egoire Meunier}
\address{École polytechnique\\91128 Palaiseau Cedex\\France}
\email{grgoir.meunier@gmail.com}

\author{Sean O'Rourke}
\address{Department of Mathematics\\ University of Colorado\\ Campus Box 395\\ Boulder, CO 80309-0395\\USA}
\email{sean.d.orourke@colorado.edu}
\thanks{S. O'Rourke has been partially supported by NSF CAREER grant DMS-2143142. }

\begin{abstract}
The Watts--Strogatz random graph model on $n$ vertices with parameters $K$ (a positive even integer) and $p \in [0, 1]$ is constructed in two steps.  First, one starts with a ring lattice on $n$ vertices, where each vertex is connected to its $K/2$ nearest neighbors on each side. Each edge in turn is then independently rewired with probability $p$ by replacing one endpoint with a uniformly chosen vertex not already adjacent to it.  
We study the empirical eigenvalue distribution of the adjacency matrix for this model, whose entries are highly dependent due to the rewiring construction.  
In the regime where both $K$ and $pK$ grow to infinity with the vertex size $n$, we show that, after appropriate scaling, the empirical eigenvalue distribution converges to the semicircle law.
The proof is based on a novel coupling argument that approximates the adjacency matrix by a sum of two independent random matrices, one a sparse Wigner matrix and the other a random band matrix.  
In the case where $K$ and $p$ remain fixed, we propose conjectural formulas for the first five moments of the limiting eigenvalue distribution. These conjectures are supported by a convergence result relating the Watts--Strogatz model to another random graph model, together with numerical simulations.
\end{abstract}

\maketitle

\section{Introduction and main results}

A small-world network is a graph in which most nodes are not neighbors but can be reached from any other node via a short path.  These networks are characterized by a high local clustering coefficient (i.e., how close a neighborhood of a node is to being complete) and small average path lengths (i.e., the average shortest path between two nodes chosen uniformly at random).
The film actor network and some neural networks appear to be empirical examples of small-world networks as shown by Watts and Strogatz \cite{watts1998collective}. In the same paper, they introduce a random graph model aimed at modeling small-world networks. This model is also the main topic of Watts's book \cite{watts1999smallworlds} and is currently referred to in the literature as the \emph{Watts--Strogatz graph model}. It is the main object of this paper.

The Watts--Strogatz random graph model on $n$ vertices with parameters $K$ (a positive even integer) and $p \in [0, 1]$ is constructed as follows. Start with a ring lattice of size $n$, in which each node is connected to $K$ neighbors ($K/2$ on each side). For each node, rewire independently with probability $p$ each of its $K/2$ right neighbors to another node chosen uniformly at random. This is made formal in Definition~\ref{def:wattsstrogatz} (see also Figure \ref{fig:graphdef}).

The Watts--Strogatz graph model has been studied in the computer science and physics literature; see, for example, \cite{barrat2000properties,barthelemy1999small,Mishra2021Multifractal,MR4455346,Newman2000MeanFieldSmallWorld,RisauGusman2004Properties} and references therein. 
In particular, Barrat and Weigt \cite{barrat2000properties} compute estimates of the asymptotic degree distribution, geodesic distance, and clustering coefficient. Several papers also numerically study the eigenvalues of the adjacency matrix of a Watts--Strogatz random graph. In particular, Risau-Gusman \cite{RisauGusman2004Properties} estimates the eigenvalue distribution in the case when $K$ is comparable with $n$, and Mishra, Raghav and Jalan \cite{MR4455346} numerically investigate the ratio of consecutive eigenvalue spacings for different graph models, including some results for the Watts--Strogatz model. Finally, a more recent paper by Mishra, Bandyopadhyay, and Jalan \cite{Mishra2021Multifractal} computes several characteristics of the eigenvectors without establishing any properties of the associated eigenvalue distribution.

However, to the best of our knowledge, the question of the exact asymptotic eigenvalue distribution of the adjacency matrix has not yet been answered. Indeed, the main contribution to the analysis of the Watts--Strogatz model in the mathematical literature comes from Alimohammadi, Isik, and Saberi \cite{Alimohammadi2025LocalLimitsSmallWorld} who show that the Watts--Strogatz random graph converges, in the sense of weak convergence of distributions on the set of rooted random graphs, to another random graph, called the full $K$-fuzz random graph, from which it is easier to deduce mathematical properties.

The only direct contribution to the study of the eigenvalues that the authors are aware of comes from Nakkirt \cite{nakkirt2020eigenvalues}, who computes the first three limiting moments of the eigenvalue distribution of the adjacency matrix of a random Watts--Strogatz graph when $p$ and $K$ are fixed. The third moment does not correspond to any known distribution, which highlights the uniqueness of the model. Indeed, the lack of independence between edges makes it hard to use any traditional method to compute the moments.

\subsection{Our contribution}

In this paper, we analytically establish the limiting empirical spectral distribution for the eigenvalues of the adjacency matrix in certain regimes and use numerical simulations to make conjectures for the other cases.  

In the regime where the product $pK$ grows to infinity but such that $K/n$ goes to $0$ as $n$ grows to infinity, our results show that the correlation between entries vanishes, and we recover the same limiting empirical eigenvalue distribution as a Wigner random matrix (i.e., a symmetric random matrix with independent entries, up to the symmetry constraint). As one could expect, the resulting eigenvalue distribution, when scaled appropriately, converges to the Wigner semicircle law (see Theorem \ref{thm:mainresult}). This is illustrated in Figure~\ref{fig:simulationdensity} with two simulations with the same values of $p$  and $n$ and different values of $K$, averaged over $10$ runs.

The Wigner semicircle law is the probability measure, denoted $\mu_{\mathrm{SC}}$, with the following density with respect to the Lebesgue measure on $\mathbb{R}$: 
\begin{equation}\label{eq:semicircledensity}
    \rho_{\mathrm{SC}}(x) =
    \begin{cases}
        \frac{1}{2\pi}\sqrt{4 - x^2} & \mbox{if } |x| \le 2, \\
        0 & \mbox{otherwise}.
    \end{cases}
\end{equation}
The semicircle law arises naturally in random matrix theory.  We refer the reader to the original works of Wigner~\cite{wigner1958roots,MR77805} as well as the texts \cite{anderson2010introduction, bai_silverstein_2010, MR2906465} for a broad overview of random matrix theory and additional references. 

In the case when $pK$ converges to $0$, a straightforward proposition allows us to show that the empirical eigenvalue distribution converges to the Dirac mass at $0$.

Additionally, we use the weak convergence result of Alimohammadi, Isik, and Saberi \cite{Alimohammadi2025LocalLimitsSmallWorld} to conjecture formulas for the first five limiting moments of the eigenvalue distribution in the regime where $K$ and $p$ are fixed. In particular, the first three conjectured moments correspond to those found by Nakkirt \cite{nakkirt2020eigenvalues}. However, the procedure we use to obtain conjectured analytical formulas for the first limiting moments does not naturally extend to arbitrarily large moments.

A few questions remain open in the critical regime when $Kp$ converges to a constant $C \in (0,\infty)$ while $K$ grows to infinity but at a slower rate than $n$. Simulations (illustrated in Figure~\ref{fig:kptoconstant}) suggest that, in the critical regime, the limiting eigenvalue distribution exists but does not follow any known distribution. These simulations also suggest that, as $C \to \infty$, when appropriately scaled, we recover the semicircle law.

\begin{figure}[htbp]
    \centering
    \begin{subfigure}{0.48\textwidth}
        \centering
        \includegraphics[width=\textwidth]{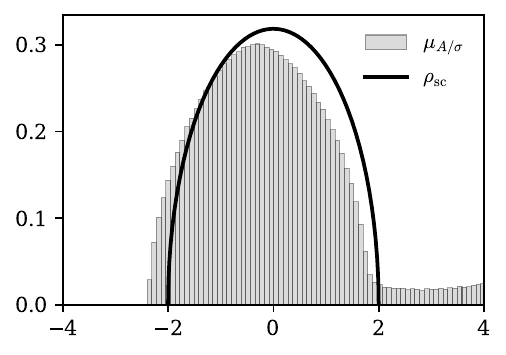}
        \caption{$n=6000$, $K=20$, $p=0.3$, averaged over $10$ runs}
        \label{fig:left}
    \end{subfigure}
    \hfill
    \begin{subfigure}{0.48\textwidth}
        \centering
        \includegraphics[width=\textwidth]{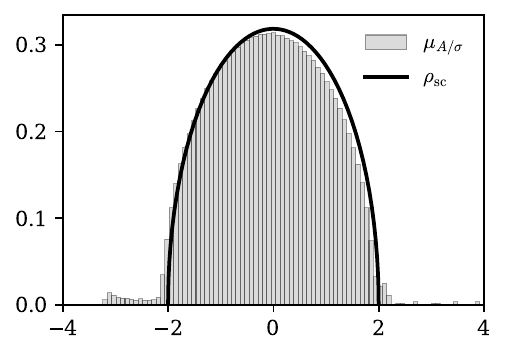}
        \caption{$n=6000$, $K=160$, $p=0.3$, averaged over $10$ runs}
        \label{fig:right}
    \end{subfigure}

    \caption{Empirical spectral density of $A/\sigma$ with $\sigma=\sqrt{Kp(2-p)}$ and $A$ the adjacency matrix of a Watts--Strogatz random graph, compared with the semicircle law. The histogram represents the empirical eigenvalue distribution, while the solid curve corresponds to the theoretical limiting density.}
    \label{fig:simulationdensity}
\end{figure}

\subsection{Definition of the Watts--Strogatz graph model}

To state our main results, we need to formally define the Watts--Strogatz graph model introduced in \cite{watts1998collective}. We use a slightly different construction, very common in the literature (see, e.g., \cite{barrat2000properties,farkas2001spectra,nakkirt2020eigenvalues}). 

\begin{defi}[Watts--Strogatz random graph model]\label{def:wattsstrogatz}
Let $n > 1$ be an integer, $K$ a positive even integer with $K<n$, and $p \in [0,1]$. The Watts--Strogatz random graph of size $n$ with parameters $(K,p)$, denoted $WS_n(K,p)$, is constructed as follows:

\begin{enumerate}
    \item \textbf{Initial lattice.}  
    Let $V_n = \{1,2,\dots,n\}$ denote the set of vertices of $WS_n(K,p)$. Arrange the vertices on a clockwise cycle. Each vertex $i \in V_n$ is connected to its $K/2$ nearest neighbors on each side, forming a regular ring lattice of degree $K$.

    \item\label{item:defWS2} \textbf{Rewiring step.}  
    For $i = 1,2,\dots,n$ in increasing order, and for each edge $\{i,j\}$ such that $j$ is one of the $K/2$ rightmost neighbors of $i$ (taken in increasing order), independently with probability $p$, rewire the edge $\{i,j\}$ as follows: remove the edge $\{i,j\}$ and add the edge $\{i,k\}$, where $k$ is chosen uniformly at random from the set of vertices in $V_n \setminus \{i,j\}$ that are not adjacent to $i$.
\end{enumerate}
\end{defi}

Since the number of edges of the Watts--Strogatz random graph is always equal to $n \cdot \frac K2$, $K$ is referred to as the mean degree.
Figure \ref{fig:graphdef} illustrates an example of a Watts--Strogatz random graph on $10$ vertices with mean degree $4$.

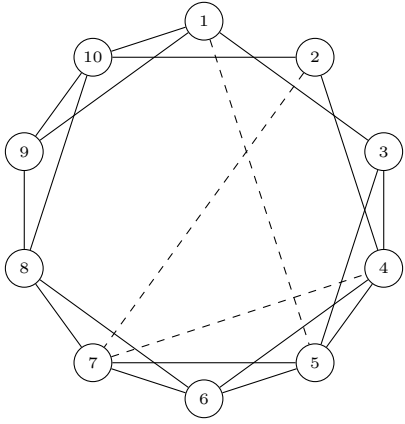
\begin{figure}[htbp]
\centering
\begin{tikzpicture}[
    every node/.style={
        circle,
        draw,
        minimum size=0.5cm,
        inner sep=1pt,
        font=\tiny
    },
    edge/.style={black}
]

\def\n{10}
\def\r{2.5}

\foreach \i in {1,...,\n}{
    \node[fill=white] (q\i) at ({90-360/\n*(\i-1)}:\r) {\i};
}

\foreach \i in {1,...,\n}{
    \pgfmathtruncatemacro{\j}{mod(\i,\n)+1}

    \ifthenelse{
        \(\i=1 \AND \j=2\) \OR
        \(\i=2 \AND \j=3\)
    }{
    }{
        \draw[edge] (q\i) -- (q\j);
    }
}

\foreach \i in {1,...,\n}{
    \pgfmathtruncatemacro{\j}{mod(\i+1,\n)+1}

    \ifthenelse{
        \(\i=7 \AND \j=9\)
    }{
    }{
        \draw[edge] (q\i) -- (q\j);
    }
}

\draw[dashed] (q1)--(q5);
\draw[dashed] (q2)--(q7);
\draw[dashed] (q7)--(q4);

\end{tikzpicture}
\caption{A Watts--Strogatz random graph on 10 vertices with mean degree 4. The edges $\{1,2\}$, $\{2,3\}$, and $\{7,9\}$ were rewired to $\{1,5\}$, $\{2,7\}$, and $\{7,4\}$, respectively.}
\label{fig:graphdef}
\end{figure}

\subsection{Main result} We begin by introducing some notation. For a real symmetric $n\times n$ matrix $M$, let $\mu_M$ denote its empirical spectral distribution:
\[
    \mu_M = \frac1n\sum_{i=1}^n\delta_{\lambda_i},
\]
where $\lambda_1,\dots,\lambda_n$ are the eigenvalues of $M$, counted with multiplicity, and $\delta_x$ denotes the Dirac mass at $x\in\mathbb{R}$.
Since the mean degree $K$ and the rewiring probability $p$ vary with $n$, we consider sequences $(K_n)_{n\in\mathbb N}$ and $(p_n)_{n\in\mathbb N}$, and study the adjacency matrix of $WS_n(K_n,p_n)$ for sufficiently large values of $n$.
Finally, for $q\in[0,1]$, each occurrence of $\bernoulli(q)$ denotes an independent Bernoulli random variable with parameter $q$.

Our main result is the following.

\begin{thm}\label{thm:mainresult}Let $(K_n)_{n\ge1}$ be a sequence of even integers satisfying $0< K_n < n$ for all sufficiently large values of $n$, and let $(p_n)_{n\ge1}$ be a sequence in $(0,1]$. Let $A_n$ denote the adjacency matrix of $WS_n(K_n, p_n)$. If
\begin{equation}\label{eq:assumption1}
    p_nK_n\toninfty\infty
\end{equation}
and
\begin{equation}\label{eq:assumption2}
    \frac{K_n}{n}\toninfty0
\end{equation}
then the empirical spectral distribution $\mu_{\frac{A_n}{\sqrt{K_n p_n(2- p_n)}}}$ of the normalized adjacency matrix $\frac{A_n}{\sqrt{K_n p_n(2- p_n)}}$ converges weakly in probability to the semicircle law $\mu_{\mathrm{SC}}$ as $n \to \infty$. 
\end{thm}

While our main result shows that, after normalization, the empirical spectral measure of $A_n$ converges to the semicircle law, we emphasize that $A_n$ is not a Wigner matrix. 
In fact, the entries of $A_n$ are quite dependent on one another due to the rewiring step in the graph construction. 
Due to this dependence, standard random matrix theory techniques, such as the moment method or Schur's complement for the resolvent, do not extend straightforwardly to this model.
Instead, the proof, given in Section~\ref{sec:proofmainresult}, relies crucially on the following comparison result.

\begin{thm}\label{thm:coupling}With the same notation as in Theorem~\ref{thm:mainresult}, define the distributions on $Z_n$ and $U_n$ as follows:
\begin{enumerate}
    \item $Z_n$ is an $n\times n$ symmetric random matrix whose upper-triangular entries are independent, and for every $i \le j$,
    \[
        (Z_n)_{ij} =
        \begin{cases}
            \bernoulli(1 - p_n) & \mbox{if }\quad 0 < \min\big(|i-j|,\, n - |i-j|\big) \le \frac{K_n}{2}, \\
            0 & \mbox{otherwise.}
        \end{cases}
    \]

    \item $U_n$ is the adjacency matrix of an Erd\H{o}s--R\'enyi random graph of size $n$ with parameter $\frac{p_n K_n}{n}$, independent of $Z_n$. In other words, $U_n$ is an $n\times n$ symmetric random matrix independent of $Z_n$ and whose upper-triangular entries are independent and satisfy, for every $i \le j$,
    \[
        (U_n)_{ij} =
        \begin{cases}
            \bernoulli\left(\frac{p_nK_n}{n}\right) & \mbox{if }\quad i < j, \\
            0 & \mbox{if } i = j.
        \end{cases}
    \]
\end{enumerate}
Let $A_n$ denote the adjacency matrix of $WS_n(K_n, p_n)$. 
Then, under the same assumptions as in Theorem~\ref{thm:mainresult}, namely~\eqref{eq:assumption1} and~\eqref{eq:assumption2}, there is a coupling of $A_n$ and $(Z_n, U_n)$ so that $Z_n$ and $U_n$ are independent and
\begin{equation} \label{eq:diffconv}
    \mu_{\frac{A_n}{\sqrt{p_nK_n}}} - \mu_{\frac{Z_n}{\sqrt{p_nK_n}} + \frac{U_n}{\sqrt{p_nK_n}}}\toninfty 0\qquad\mbox{ weakly in probability}. 
\end{equation}
\end{thm}

Theorem~\ref{thm:coupling} allows us to reduce the proof of Theorem~\ref{thm:mainresult} to the study of the sum of independent random matrices $\frac{Z_n}{\sqrt{p_nK_n}} + \frac{U_n}{\sqrt{p_nK_n}}$, for which we can apply standard random matrix theory results. 
The proof of Theorem~\ref{thm:coupling} relies on a novel coupling argument and is given in Section~\ref{sec:coupling}.

Finally, the following proposition, whose proof is given in Appendix~\ref{sec:kpto0}, describes the case when $p_nK_n \to 0$ as $n \to \infty$.
\begin{prop}\label{prop:kpto0}
    With the same notation as in Theorem~\ref{thm:mainresult}, assume
    \[
        p_nK_n\toninfty 0\quad\mbox{ and }\quad K_n/n\toninfty0,
    \]with no additional assumption on the asymptotic behaviors of $(K_n)_{n\ge1}$ and $(p_n)_{n\ge1}$.

    \noindent Then,
    \begin{equation}
        \mu_{\frac{A_n}{\sqrt{p_nK_n}}}\toninfty\delta_0\qquad\mbox{ weakly in probability},
    \end{equation}where $\delta_0$ denotes the Dirac mass at $0$.
\end{prop}

\subsection{Discussion and conjectures}








Our main results handle the cases when $K_n p_n \to \infty$ and $K_n p_n \to 0$.  It is natural to also consider the case when $K_n p_n \to C \in (0, \infty)$.
\begin{conj} \label{conj:C}
    In the regime where $K_n\toninfty\infty$ with $K_n/n\toninfty0$ while $K_np_n\toninfty C$ for a constant $C\in(0,\infty)$, we conjecture that there exists a measure $\mu_C$ that depends only on $C$ such that
\[
    \mu_{\frac{A_n}{\sqrt{K_np_n(2-p_n)}}}\toninfty\mu_C\qquad\mbox{ weakly in probability},
\]where $A_n$ denotes the adjacency matrix of $WS_n(K_n,p_n)$. 
\end{conj}
Under the hypotheses of Conjecture \ref{conj:C}, Remark~\ref{rk:muctomusc} below makes it natural to conjecture that
\[
    \mu_C\;\xrightarrow[C\to\infty]{}\;\mu_{\mathrm{SC}}, 
\]where the convergence is the weak convergence of probability measures. We performed simulations for various values of $K$ and $p$, which confirm these conjectures. Figure~\ref{fig:kptoconstant} shows the empirical distributions obtained from these simulations, for matrices of size $3500$, averaged over $4$ independent runs.

When $K$ and $p$ are fixed, we make in Section~\ref{sec:limitingmoments} a conjecture about the limiting moments of the empirical eigenvalue distribution of the adjacency matrix of $WS_n(K,p)$. Under the assumption that this conjecture holds, we derive some closed analytic formulas for the first five limiting moments. To the best of our knowledge, these conjectured moments do not correspond to any known distribution.
Under the natural hypothesis that the empirical distribution of the adjacency matrix of $WS_n(K,p)$, when scaled by $\sqrt{Kp(2-p)}$, converges to a limiting measure $\mu_{K,p}$ on the real line, Remark~\ref{rk:muctomusc} below motivates the following conjecture.
\begin{conj}
    $
    \mu_{K,p}\;\xrightarrow[Kp\to\infty]{}\;\mu_{\mathrm{SC}},
    $
    weakly in probability.
\end{conj}

\begin{figure}[htbp]
    \centering
    
    \begin{minipage}{0.3\textwidth}
    \centering
    \textbf{$Kp = 0.05$}
    \end{minipage}
    \hfill
    \begin{minipage}{0.3\textwidth}
    \centering
    \textbf{$Kp = 0.5$}
    \end{minipage}
    \hfill
    \begin{minipage}{0.3\textwidth}
    \centering
    \textbf{$Kp = 5$}
    \end{minipage}
    
    \vspace{0.5em}
    
    \begin{subfigure}{0.3\textwidth}
        \centering
        \includegraphics[width=\textwidth]{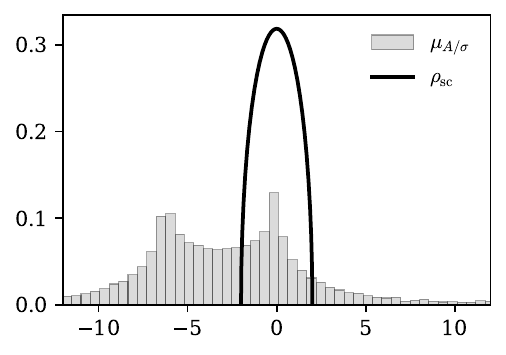}
        \caption{$K=200$, $p=2.5\times10^{-4}$}
    \end{subfigure}
    \hfill
    \begin{subfigure}{0.3\textwidth}
        \centering
        \includegraphics[width=\textwidth]{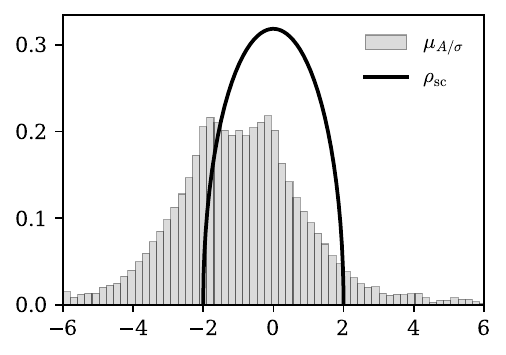}
        \caption{$K=200$, $p=2.5\times10^{-3}$}
    \end{subfigure}
    \hfill
    \begin{subfigure}{0.3\textwidth}
        \centering
        \includegraphics[width=\textwidth]{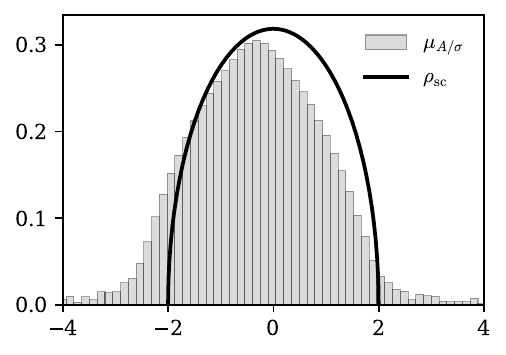}
        \caption{$K=200$, $p=2.5\times10^{-2}$}
    \end{subfigure}
    
    \vspace{0.5em}
    
    \begin{subfigure}{0.3\textwidth}
        \centering
        \includegraphics[width=\textwidth]{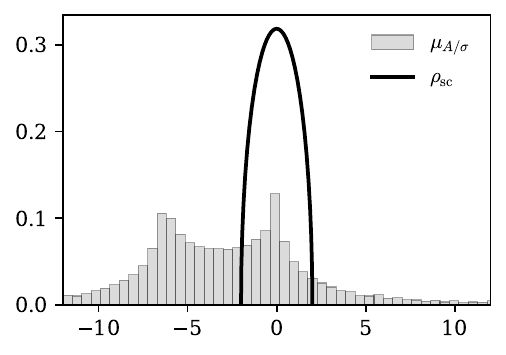}
        \caption{$K=400$, $p=1.25\times10^{-4}$}
    \end{subfigure}
    \hfill
    \begin{subfigure}{0.3\textwidth}
        \centering
        \includegraphics[width=\textwidth]{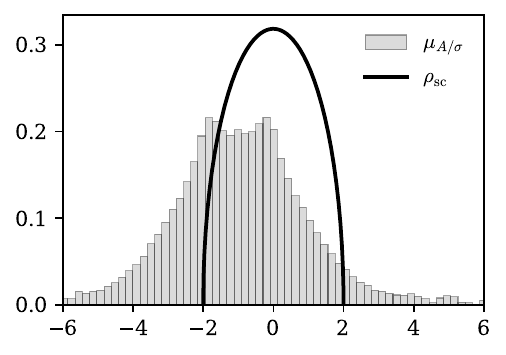}
        \caption{$K=400$, $p=1.25\times10^{-3}$}
    \end{subfigure}
    \hfill
    \begin{subfigure}{0.3\textwidth}
        \centering
        \includegraphics[width=\textwidth]{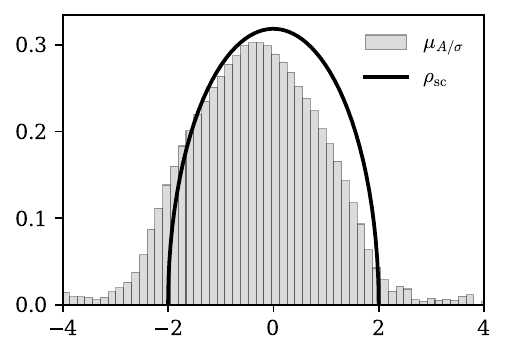}
        \caption{$K=400$, $p=1.25\times10^{-2}$}
    \end{subfigure}
    
    \caption{Empirical spectral density of $A/\sigma$ with $\sigma=\sqrt{Kp(2-p)}$ and $A$ the adjacency matrix of $WS_n(K,p)$, compared with the semicircle law. The histogram represents the empirical eigenvalue distribution, while the solid curve corresponds to the theoretical limiting density. Columns correspond to fixed values of $Kp$, while rows vary $K$.}
    \label{fig:kptoconstant}
\end{figure}

\subsection{Related work}
While we are not aware of any prior results proving the convergence of the eigenvalue distribution of the Watts--Strogatz model, we mention in this subsection results concerning some related random matrix models and random graph models.

It is well known that, under mild assumptions, the empirical eigenvalue distribution of symmetric matrices with independent and identically distributed (i.i.d.) upper-triangular entries, converges (in an appropriate sense) to the Wigner semicircle law (see e.g. \cite{erdos2011universality}). This result extends to several related models. In particular, it is shown in  \cite{MolchanovPasturKhorunzhii1992, BMP1991} that, under additional weak assumptions, the universality of the Wigner semicircle law holds for symmetric real band matrices with i.i.d. entries. It is worth mentioning that our result relies on this universality property.

A classical framework for random matrices arises from adjacency matrices of random graphs.
The most studied model of random graphs is the Erd\H{o}s--R\'enyi random graph, which starts with a complete graph where each edge is independently removed with the same probability. Many results are known for the eigenvalues of the adjacency matrix of this model, including the semicircle law \cite{MR637828, MR77805, wigner1958roots} and local laws  \cite{erdos2012graph2,erdos2013graph1}.  

The universality of the Wigner semicircle law typically fails when correlations are introduced. For instance, for regular graphs, i.e., graphs chosen uniformly at random among graphs with fixed degree $k$, McKay \cite{mckay1981regular} derives an explicit formula for the limiting eigenvalue distribution, now known as the McKay distribution.  We  also refer the reader to \cite{tran_vu_wang_2013} for the case when $k \to \infty$. 

It is also worth mentioning the power-law graph model. When the degree distribution follows a power law, Chung, Lu and Vu \cite{chung2003expecteddegrees} prove convergence results for the empirical spectral distribution. Finally, Bordenave and Lelarge \cite{bordenave_lelarge_2010} characterize the limiting spectral distribution in the case where the graph locally converges to a Galton--Watson tree with prescribed degree distribution.

\subsection{Overview of the paper}
Section~\ref{sec:notation} recalls and introduces notation used in the paper. Sections~\ref{sec:coupling} and~\ref{sec:proofmainresult} are dedicated to the proofs of Theorem~\ref{thm:coupling} and Theorem~\ref{thm:mainresult}, respectively. We discuss in Section~\ref{sec:limitingmoments} the case where $K$ and $p$ are fixed. Finally, Appendix~\ref{sec:tools} recalls a few well-known results from Bai and Silverstein's text \cite{bai_silverstein_2010} used throughout the paper, and Appendix~\ref{sec:kpto0} gives the proof of Proposition~\ref{prop:kpto0}.
Appendix~\ref{sec:replacement} provides a replacement principle result (Theorem \ref{thm:replacement}) used in several parts of the proof.

\section{Notation}\label{sec:notation}

We recall and introduce the following notation that is used throughout the paper.

For a matrix $M$, $\nnz(M)$ denotes the number of nonzero entries in $M$.  In particular, it always follows that $\rank(M) \leq \nnz(M)$.  $\|M\|$ is the spectral norm of $M$. Namely, when $M$ is an $n \times n$ real symmetric matrix, 
\[
    \|M\| = \max_{i\in\{1,\dots,n\}}|\lambda_i|,
\]where $\lambda_1,\dots,\lambda_n$ are the eigenvalues of $M$, counted with multiplicity.

For a real symmetric $n\times n$ matrix $M$, $\mu_M$ denotes its empirical spectral distribution, defined by
\[
    \mu_M = \frac1n\sum_{i=1}^n\delta_{\lambda_i},
\]where $\lambda_1,\dots,\lambda_n$ are the eigenvalues of $M$, counted with multiplicity, and $\delta_x$ is the Dirac mass at $x \in \mathbb{R}$.

For $n\in\mathbb{N}$ and real numbers $\lambda_1,\dots,\lambda_n$, we denote by $\diag(\lambda_1,\dots,\lambda_n)$ the diagonal matrix with diagonal entries $\lambda_1,\dots,\lambda_n$:
\[
    \diag(\lambda_1,\dots,\lambda_n)=
    \begin{pmatrix}
        \lambda_1 & 0 & \cdots & 0\\
        0 & \lambda_2 & \ddots & \vdots\\
        \vdots & \ddots & \ddots & 0\\
        0 & \cdots & 0 & \lambda_n
    \end{pmatrix}.
\]

$\mu_{\mathrm{SC}}$ always denotes the semicircle law, in the sense of the measure on $\mathbb{R}$ with density $\rho_{\mathrm{SC}}$ given in \eqref{eq:semicircledensity}. 

For $k\in\mathbb{N}$ and $(M_n)_{n\in\mathbb{N}}$ a matrix sequence such that $\lim_{n\to\infty}\frac{1}{n}\mathbb{E}[\tr(M_n^k)]$ exists, we write $\tau(M_n^k) =\lim_{n\to\infty}\frac{1}{n}\mathbb{E}[\tr(M_n^k)]$ for the limiting moments.

For a finite set $S$, $|S|$ denotes its cardinality. If $V \subseteq S$, then $V^c$ denotes the complement of $V$ in $S$.

For all $q \in [0,1]$, each occurrence of $\bernoulli(q)$ denotes an independent copy of a Bernoulli random variable with parameter $q$. Similarly, each occurrence of $\poisson(\lambda)$, for $\lambda > 0$, denotes an independent copy of a Poisson random variable with parameter $\lambda$. We use $\binomial(n,q)$ for a Binomial random variable with parameters $n \ge 1$ and $q \in [0,1]$, and $\mathcal{N}(m,\sigma^2)$ for a normal random variable with mean $m \in \mathbb{R}$ and variance $\sigma^2 > 0$.

Unless otherwise stated, we use asymptotic notation under the assumption that $n \to \infty$.  We write $a_n = o(b_n)$ if \[ \lim_{n \to \infty} \frac{a_n}{b_n} = 0. \]

\section{Proof of Theorem~\ref{thm:coupling}}\label{sec:coupling}
Let $(K_n)_{n\ge1}$ denote a sequence of even numbers that satisfy $0<K_n<n$ for all sufficiently large values of $n$, and let $(p_n)_{n\ge1}$ denote a sequence of numbers in $(0,1]$.

The proof relies entirely on the construction, for all $n$ large enough, of three $n \times n$ matrices $U_n$, $Z_n$, and $\widetilde{U}_n$ that satisfy the following proposition.
\begin{prop}\label{prop:coupling} 
Let $A_n$ be the adjacency matrix of $WS_n(K_n,p_n)$.  There is a coupling between $A_n$ and $(U_n, Z_n, \widetilde{U}_n)$ so that the following properties hold.
\begin{enumerate}[itemsep=0.3em]
    \item \label{item:AequalZU}$A_n = Z_n+\widetilde{U}_n$. 
    \item \label{item:ZUditrib} $Z_n$ and $U_n$ are independent and have the same distributions as stated in Theorem~\ref{thm:coupling}. 
    \item \label{item:secondsm} We have
    \[
        \mathbb{E}\left[ \frac{1}{n} \tr \left( \frac{(Z_n+U_n) - (Z_n+\widetilde{U}_n)}{\sqrt{p_nK_n}} \right)^2 \right] \le 8\left(\frac{K_n}{n}+\frac{1}{\sqrt{p_nK_n}}\right).
    \]
\end{enumerate}
\end{prop}

\begin{rk}\label{rk:muctomusc}
    Note that Proposition~\ref{prop:coupling} does not rely on any restrictive assumptions on $(K_n)_{n\ge1}$ or $(p_n)_{n\ge1}$. Therefore, if $K_np_n\underset{n\to\infty}{\longrightarrow} C\in(0,\infty)$, the second limiting spectral moment of $\frac{(Z_n+U_n) - (Z_n+\widetilde{U}_n)}{\sqrt{p_nK_n}}$ is bounded by $8/\sqrt{C}$, which goes to $0$ as $C\to\infty$. Therefore, since $U_n$ and $Z_n$ are generalized Wigner matrices, it is natural to conjecture that, if $Z_n+\widetilde{U}_n$ admits a limiting spectral distribution, this limit converges to $\mu_{\mathrm{SC}}$ as $C\to\infty$.
\end{rk}

We are now ready to prove Theorem \ref{thm:coupling}. 
\begin{proof}[Proof of Theorem \ref{thm:coupling}]
    Using part \eqref{item:secondsm} of Proposition~\ref{prop:coupling}, under assumptions \eqref{eq:assumption1} and \eqref{eq:assumption2} from Theorem~\ref{thm:coupling}, the second spectral moment of the difference between $\frac{{Z}_n+\widetilde{U}_n}{\sqrt{p_nK_n}}$ and $\frac{{Z}_n+{U}_n}{\sqrt{p_nK_n}}$ goes to $0$. Hence by standard bounds (see, for instance, \cite[Corollary A.41]{bai_silverstein_2010}, recalled as Theorem~\ref{thm:cor41} in Appendix~\ref{sec:tools}) and Markov's inequality
\[
    \mu_{\frac{Z_n+\widetilde{U}_n}{\sqrt{K_np_n}}}-\mu_{\frac{{Z}_n+{U}_n}{\sqrt{K_np_n}}}\toninfty0\qquad\mbox{ weakly in probability}.
\]
From Proposition \ref{prop:coupling}, $A _n = Z_n+\widetilde{U}_n$, and the claim follows.
\end{proof}

The rest of the section is devoted to the proof of Proposition \ref{prop:coupling}.  The proof is structured in two subsections. The first subsection gives the construction of $U_n$, $Z_n$, and $\widetilde{U}_n$. The second one proves, based on this construction, the properties claimed in Proposition~\ref{prop:coupling}.

\begin{figure}[htbp]
\centering
\resizebox{!}{0.95\textheight}{%
\begin{tikzpicture}[every node/.style={circle,draw,minimum size=0.7cm}]

\node[draw=none, anchor=east, font=\large] at (-3.3,7) {\underline{Initialization:}};
\begin{scope}[xshift=0cm,yshift=7cm]
\foreach \i in {1,...,8}
    \node (q\i) at ({90-45*(\i-1)}:2) {\i};

\node[draw=none, text width=7cm, align=center] at (0,-3.5) {$\mathcal{E}$, initially defined as an empty graph.};
\end{scope}

\begin{scope}[xshift=8cm,yshift=7cm]
\foreach \i in {1,...,8}
    \node (r\i) at ({90-45*(\i-1)}:2) {\i};

\foreach \i in {1,...,8}{
    \pgfmathtruncatemacro{\next}{mod(\i,8)+1}
    \pgfmathtruncatemacro{\nextnext}{mod(\i+1,8)+1}
    \draw (r\i) -- (r\next);
    \draw (r\i) -- (r\nextnext);
}

\node[draw=none, text width=7cm, align=center] at (0,-3.5) {$\mathcal{W}$, initially defined as a regular ring lattice.\\ Edges belonging to $\mathcal{Z}$ are solid, those belonging to $\mathcal{U}$ are dashed.};
\end{scope}

\node[draw=none, anchor=east, font=\large] at (-4,0) {\underline{Node 1:}};
\begin{scope}[xshift=0cm,yshift=0cm]
\foreach \i in {1,...,8}
    \node (q\i) at ({90-45*(\i-1)}:2) {\i};

\draw (q1) -- (q4);
\draw (q1) -- (q7);

\node[draw=none, text width=7cm, align=center] at (0,-3.5) {$\mathcal{C}^{1}=\{4,7\}$.\\
Nodes 4 and 7 have been selected as candidate neighbors, each with probability $1-\sqrt{1-\frac{p_n K_n}{n}}$.};
\end{scope}

\begin{scope}[xshift=8cm,yshift=0cm]
\foreach \i in {1,...,8}
    \node (r\i) at ({90-45*(\i-1)}:2) {\i};

\foreach \i in {1,...,8}{
    \pgfmathtruncatemacro{\next}{mod(\i,8)+1}
    \pgfmathtruncatemacro{\nextnext}{mod(\i+1,8)+1}
    \draw (r\i) -- (r\next);
    \draw (r\i) -- (r\nextnext);
}
\draw[white, thick=4pt] (r1) -- (r3);
\draw[dashed] (r1) -- (r4);

\node[draw=none, text width=7cm, align=center] at (0,-3.5) {Edge $\{1,2\}$ is kept (probability $(1-p)$), while $\{1,3\}$ is cut off (probability $p$).\\The replacing neighbor $4$ is selected uniformly at random from $\{4,5,6\}\cap\mathcal{C}^{1}$.};
\end{scope}

\node[draw=none, anchor=east, font=\large] at (-4,-7) {\underline{Node 2:}};
\begin{scope}[xshift=0cm,yshift=-7cm]
\foreach \i in {1,...,8}
    \node (q\i) at ({90-45*(\i-1)}:2) {\i};

\draw (q1) -- (q2);
\draw (q1) -- (q4);
\draw (q1) -- (q7);
\draw (q2) -- (q5);

\node[draw=none, text width=7cm, align=center] at (0,-3.5) {$\mathcal{C}^{2} = \{1,5\}$.\\
Edge $\{1,2\}$ has now had two opportunities to be added to $\mathcal{E}$, with total probability $\frac{p_nK_n}{n}$.};
\end{scope}

\begin{scope}[xshift=8cm,yshift=-7cm]
\foreach \i in {1,...,8}
    \node (r\i) at ({90-45*(\i-1)}:2) {\i};

\foreach \i in {1,...,8}{
    \pgfmathtruncatemacro{\next}{mod(\i,8)+1}
    \pgfmathtruncatemacro{\nextnext}{mod(\i+1,8)+1}
    \draw (r\i) -- (r\next);
    \draw (r\i) -- (r\nextnext);
}
\draw[white, thick=4pt] (r1) -- (r3);
\draw[dashed] (r1) -- (r4);

\draw[white, thick=4pt] (r2)--(r3);
\draw[white, thick=4pt] (r2)--(r4);

\draw[dashed] (r2)--(r5);
\draw[dashed] (r2)--(r7);

\node[draw=none, text width=7cm, align=center] at (0,-3.5) {Edge $\{2,3\}$ is cut off and rewired to $\{2,5\}$. Edge $\{2,4\}$ is also removed; since the only remaining candidate is $1$, already adjacent to $2$, it is instead rewired to $7$, chosen uniformly at random from $\{3,6,7\}$.};
\end{scope}

\end{tikzpicture}
}

\caption{Illustration of the graphs $\mathcal{E}$ (left) and $\mathcal{W}$ (right) during the construction for $n=8$ and $K_n=4$.}
\label{fig:construction}
\end{figure}

\subsection{Construction of $(U_n,Z_n,\widetilde{U}_n)$}
The aim of this subsection is to construct the three matrices $U_n$, $Z_n$, and $\widetilde{U}_n$ introduced in Proposition~\ref{prop:coupling}. For ease of notation, we omit the index $n$ when it is clear from the context. For instance, $U_n$, $Z_n$, and $\widetilde{U}_n$ will be denoted $U$, $Z$, and $\widetilde{U}$, respectively.

The construction procedure, described below, produces two random graphs, $\mathcal{W}$ and $\mathcal{E}$, defined on the vertex set $\{1,\dots,n\}$. The graph $\mathcal{W}$ is constructed as a Watts--Strogatz random graph (following Definition~\ref{def:wattsstrogatz}) and, accordingly, will be initialized as a regular ring lattice. $\mathcal{E}$ is constructed as an Erd\H{o}s--R\'enyi random graph of size $n$ with parameter $\frac{p_nK_n}{n}$ (in other words, each edge will independently be present with probability $\frac{p_nK_n}{n}$). It will be initialized as the empty graph on $n$ vertices.

Additionally, we will define $\mathcal{Z}$ as an evolving subgraph of $\mathcal{W}$ that contains all vertices and all initial edges of $\mathcal{W}$. At each step of the construction, whenever an edge is removed from $\mathcal{W}$, it is also removed from $\mathcal{Z}$. We will also define $\mathcal{U}$ as an evolving subgraph of $\mathcal{W}$ that contains all vertices of $\mathcal{W}$ and whose edge set, initially empty, consists of all edges added during the construction. Hence, $\mathcal{Z}$ only loses edges, while $\mathcal{U}$ only gains edges, so that at every step of the construction, the edge set of $\mathcal{W}$ is the disjoint union of those of $\mathcal{Z}$ and $\mathcal{U}$.

The following construction of $\mathcal{W}$ and $\mathcal{E}$ will couple $\mathcal{E}$ and $\mathcal{U}$ in a way that will ensure that the number of edges that differ stays controllable.

Namely, we define the three matrices $U$, $Z$ and $\widetilde{U}$ as the adjacency matrices of $\mathcal{E}$, $\mathcal{Z}$ and $\mathcal{U}$, respectively. The construction induces a coupling between $U$ and $\widetilde{U}$ that ensures that
\[
    \nnz(U-\widetilde{U})
\]is bounded sufficiently in order to prove claim~\ref{item:secondsm} of Proposition~\ref{prop:coupling}.

Finally, we will denote by $U^{(i)}$, $Z^{(i)}$ and $\widetilde{U}^{(i)}$ the states of the matrices $U$, $Z$ and $\widetilde{U}$ after processing node $i$. Therefore, at the end of the construction, we will have
\[
    U=U^{(n)},\qquad Z = Z^{(n)}\quad\mbox{ and }\quad\widetilde{U} = \widetilde{U}^{(n)}.
\]This notation is introduced to keep track of a bound on
\[
    \nnz(U^{(i)}-\widetilde{U}^{(i)})-\nnz(U^{(i-1)}-\widetilde{U}^{(i-1)}),
\]which will be useful to derive the desired bound on $\nnz(U-\widetilde{U})$.

Throughout the construction, for the sake of clarity, when writing $\{i,j\}$ or $ij$ for an edge or a matrix entry, we will abusively allow $i$ or $j$ to be greater than $n$, in which case it has to be understood as the corresponding integer modulo $n$ in $\{1,\dots,n\}$. 

\textbf{Initialization.} \textit{(This step is illustrated in the Initialization part of Figure~\ref{fig:construction}.)} Let $\mathcal{W}$ and $\mathcal{E}$ be initialized as two graphs on the vertex set $\{1,\dots,n\}$, $\mathcal{W}$ as a regular ring lattice, that is, a graph with $n$ nodes, each connected to its $K$ closest neighbors ($K/2$ on each side), and $\mathcal{E}$ a graph with no edges. We also initialize $\mathcal{Z}$ and $\mathcal{U}$ as described above, that is: $\mathcal{Z}$ has the same node set and the same edge set as $\mathcal{W}$, while $\mathcal{U}$ has the same node set, but no edges.

Let $Z^{(0)}$, $\widetilde{U}^{(0)}$, and $U^{(0)}$ denote the adjacency matrices of $\mathcal{Z}$, $\mathcal{U}$, and $\mathcal{E}$ at this stage, respectively. In particular, $Z^{(0)}$ has entries equal to $0$ everywhere except in the extended $K$-band\footnote{We call the extended $K$-band of a matrix $M$ of size $n > K$ the set of entries $(i,j)$ such that
\[
    0 < \min\left(|i-j|,\, n - |i-j|\right) \le \frac{K}{2}.
\]Note that the extended $K$-band includes entries in the top-right and bottom-left corners.}, where its entries are equal to $1$. Moreover, we have
\[
    \widetilde{U}^{(0)} = U^{(0)} = 0.
\]

\textbf{Construction.} \textit{(This step is illustrated for Nodes $1$ and $2$ in Figure~\ref{fig:construction}.)} For each $i \in \{1,\dots,n\}$:

\begin{enumerate}
    \item Initially set
    \[
        U^{(i)} = U^{(i-1)}, \quad \widetilde{U}^{(i)} = \widetilde{U}^{(i-1)}, \quad Z^{(i)} = Z^{(i-1)}.
    \]
    \item \label{item:defineC} Define $\mathcal{C}^i$ to be the random set of candidate new neighbors of node $i$, obtained by including each node $j \in \{1,\ldots,n\}\setminus\{i\}$ independently with probability
    \[
        1-\sqrt{1-\frac{p_nK_n}{n}}.
    \]This choice of probability will ensure that every edge in $\mathcal{E}$ has an overall probability of
    \[
        \frac{p_nK_n}{n}
    \]to be present. This will be made more explicit below, in the proof of Proposition~\ref{prop:coupling}.
    \item \label{item:updateU} Define $E^i$ to be the set of nodes that are not adjacent to $i$ in $\mathcal{E}$ at this stage. For all $j\in E^i\cap\mathcal{C}^i$, add the edge $\{i,j\}$ in $\mathcal{E}$. Equivalently, for all such $j$, set
    \[
        U^{(i)}_{ij}\leftarrow1\quad\mbox{ and }\quad U^{(i)}_{ji}\leftarrow 1.
    \]
    \item\label{item:couplRewiring} This step follows step~\ref{item:defWS2} in Definition~\ref{def:wattsstrogatz}. For each $j\in\{1,\dots,K/2\}$, taken in increasing order, independently with probability $p_n$, rewire the edge $e_j=\{i,(i+j)\}$ in $\mathcal{W}$ as follows:
    \begin{enumerate}
        \item \label{item:updateZ} Remove the edge $e_j$ from $\mathcal{W}$ (and therefore from $\mathcal{Z}$). In other words, set
        \[
            Z^{(i)}_{i,(i+j)}\leftarrow0\quad\mbox{ and }\quad Z^{(i)}_{(i+j),i}\leftarrow0.
        \]
        \item\label{item:rewiring} Define $W^{ij}$ to be the set of vertices in $\{1,\dots,n\}\setminus\{i,(i+j)\}$ that are not adjacent to $i$ in $\mathcal{W}$ at this stage. Choose the index $k$ uniformly at random from $W^{ij}$ following those steps:
        \begin{enumerate}
            \item\label{item:choosekfromC} If $W^{ij}\cap\mathcal{C}^i\ne\emptyset$, choose $k$ uniformly at random from $W^{ij}\cap\mathcal{C}^i$.
            \item\label{item:choosekfromW} Otherwise, choose $k$ uniformly at random from $W^{ij}$. 
        \end{enumerate}Since every node has the same probability to be in $\mathcal{C}^i$, $k$ is indeed chosen uniformly at random from $W^{ij}$.
        \item\label{item:couplUTildaUpdate} Add edge $\{i,k\}$ to $\mathcal{W}$ (and therefore to $\mathcal{U}$). In other words, set
        \[
            \widetilde{U}^{(i)}_{ik}\leftarrow1\quad\mbox{ and }\quad \widetilde{U}^{(i)}_{ki}\leftarrow 1.
        \]
    \end{enumerate}
\end{enumerate}

\textbf{Invariant.} For all $i\in\{1,\dots,n\}$, we define the evolving quantity 
\[
    N^i = \nnz(U^{(i)}-\widetilde{U}^{(i)}) - \nnz(U^{(i-1)}-\widetilde{U}^{(i-1)}).
\] 
We now go through each step of the construction, bounding systematically the increments of $N^i$.
\begin{enumerate}
    \item Here, we have $N^i=0$.
    \item Unchanged.
    \item For all $e\in\mathcal{C}^i\cap E^i$, then, when adding $e$ to $\mathcal{E}$, we have:
    \begin{itemize}
        \item If $e$ is not present in $\mathcal{U}$, then $N^i$ increases by two.
        \item If $e$ is present in $\mathcal{U}$, then $N^i$ decreases by two.
    \end{itemize}Let $V^i$ denote the set of nodes that are not adjacent to $i$ in $\mathcal{U}$ at this step. There are $|E^i\cap\mathcal{C}^i\cap V^i|$ iterations of the first case and $|E^i\cap\mathcal{C}^i\cap (V^i)^c|$ of the second (where $(V^i)^c$ denotes the complement of $V^i$ in $\{1,\dots,n\}$). Therefore, $N^i$ increases by
    \[
        2|E^i\cap\mathcal{C}^i\cap V^i|-2|E^i\cap\mathcal{C}^i\cap (V^i)^c|.
    \]
    \item For all $j\in\{1,\dots,K/2\}$,\begin{enumerate}
        \item Unchanged.
        \item We consider that step~\ref{item:couplUTildaUpdate} (updating $\widetilde{U}$) is executed at the same time.
        \begin{enumerate}
            \item If $|W^{ij}\cap\mathcal{C}^i|\ge1$, $N_i$ decreases by $2$, because every node in $\mathcal{C}^i$ is connected to node $i$ in $\mathcal{E}$.
            \item If $|W^{ij}\cap\mathcal{C}^i|=0$, $N_i$ increases by at most $2$.
        \end{enumerate}
        \item This step was taken into account in the previous one.
    \end{enumerate}Let $r_i$ denote the number of rewiring events that occurred for node $i$ (or, equivalently, the number of executions of step~\ref{item:rewiring}). Notice that step~\ref{item:choosekfromC} is executed at least $\min\{|W^{i1}\cap\mathcal{C}^i|,r_i\}$ times, because the number of candidates not adjacent to $i$ decreases by at most $1$ every time. Therefore, step~\ref{item:choosekfromW} is executed at most $\max\{r_i-|W^{i1}\cap\mathcal{C}^i|,0\}$.
\end{enumerate}
At the end of the construction, we have, for all $i\in\{1,\dots,n\}$,
\begin{align*}
    \nnz(U^{(i)}-\widetilde{U}^{(i)}) \le &\nnz(U^{(i-1)}-\widetilde{U}^{(i-1)})\\
    &+2|E^i\cap\mathcal{C}^i\cap V^i|-2|E^i\cap\mathcal{C}^i\cap (V^i)^c|\\
    &- 2\min\{|W^{i1}\cap\mathcal{C}^i|,r_i\}\\
    &+2\max\{r_i-|W^{i1}\cap\mathcal{C}^i|,0\}.
\end{align*}Consequently, if $r_i\le|W^{i1}\cap\mathcal{C}^i|$,
\begin{align*}
    \nnz(U^{(i)}-\widetilde{U}^{(i)}) &\le \nnz(U^{(i-1)}-\widetilde{U}^{(i-1)}) +2|E^i\cap\mathcal{C}^i\cap V^i|-2|E^i\cap\mathcal{C}^i\cap(V^i)^c|- 2r_i\\
    &\le\nnz(U^{(i-1)}-\widetilde{U}^{(i-1)}) + 2||E^i\cap\mathcal{C}^i|-r_i|
\end{align*}and, if $r_i>|W^{i1}\cap\mathcal{C}^i|$,
\begin{align*}
    \nnz(U^{(i)}-\widetilde{U}^{(i)}) &\le\nnz(U^{(i-1)}-\widetilde{U}^{(i-1)}) +2|E^i\cap\mathcal{C}^i\cap V^i|-2|E^i\cap\mathcal{C}^i\cap(V^i)^c|\\
    &\hspace{2cm}- 2|W^{i1}\cap\mathcal{C}^i|+2(r_i-|W^{i1}\cap\mathcal{C}^i|).
\end{align*}In that second case, we need to compare $V^i$ with $W^{i1}$. Recall that they are equal to the sets of nodes not adjacent to $i$ in $\mathcal{U}$ and $\mathcal{W}$, respectively, when those graphs are considered at the beginning of step~\ref{item:couplRewiring}. Now, since the only edges that differ between those graphs belong to $\mathcal{Z}$, which is a subgraph of the regular ring lattice, we have
\[
    V^i\subset W^{i1}\cup I^i,
\]where $I^i$ denotes the set of initial neighbors of $i$ in $\mathcal{W}$ (in other words, the neighbors of $i$ in the regular ring lattice). In particular,
\[
    |E^i\cap \mathcal{C}^i\cap V^i|\le |\mathcal{C}^i\cap I^i|+|E^i\cap \mathcal{C}^i\cap W^{i1}|
\]and
\[
    |E^i\cap \mathcal{C}^i\cap (V^i)^c|\ge|E^i\cap \mathcal{C}^i\cap (W^{i1})^c|-|\mathcal{C}^i\cap I^i|.
\]Therefore, in the case $r_i>|W^{i1}\cap\mathcal{C}^i|$, we have
\begin{align*}
    \nnz(U^{(i)}-\widetilde{U}^{(i)}) &\le\nnz(U^{(i-1)}-\widetilde{U}^{(i-1)})\\
    &\hspace{2cm}+4|\mathcal{C}^i\cap I^i|+2|E^i\cap\mathcal{C}^i\cap W^{i1}|-2|E^i\cap\mathcal{C}^i\cap(W^{i1})^c|\\
    &\hspace{2cm}- 2|W^{i1}\cap\mathcal{C}^i|+2(r_i-|W^{i1}\cap\mathcal{C}^i|).
\end{align*}Additionally, note that $|E^i\cap\mathcal{C}^i\cap W^{i1}|-|W^{i1}\cap\mathcal{C}^i|\le 0$ and
\[
    |W^{i1}\cap\mathcal{C}^i|+|E^i\cap\mathcal{C}^i\cap (W^{i1})^c|\ge|E^i\cap\mathcal{C}^i|.
\]

Therefore, in both cases, we have the following invariant
\[
    \nnz(U^{(i)}-\widetilde{U}^{(i)})\le\nnz(U^{(i-1)}-\widetilde{U}^{(i-1)})+4|\mathcal{C}^i\cap I^i|+2||E^i\cap\mathcal{C}^i|-r_i|,
\]from which we deduce
\begin{equation}\label{eq:inv}
    \nnz(U-\widetilde{U})\le\sum_{i=1}^n(4|\mathcal{C}^i\cap I^i|+2||E^i\cap\mathcal{C}^i|-r_i|).
\end{equation}

\subsection{Conclusion of the proof of Proposition \ref{prop:coupling}}\label{subsec:proofsconstructionprop}

With the construction of $(Z_n, U_n, \widetilde{U}_n)$ above, we can now complete the proof of Proposition~\ref{prop:coupling}.

\begin{proof}[Proof of Proposition~\ref{prop:coupling}]
We verify the three claims separately.

\noindent\textit{Proof of (\ref{item:AequalZU}).} By construction, $Z_n+\widetilde{U}_n$ is the adjacency matrix of $\mathcal{W}$. It is clear that the construction of $\mathcal{W}$ follows exactly the same steps as in Definition~\ref{def:wattsstrogatz}.\\

\noindent\textit{Proof of (\ref{item:ZUditrib}).} First, $Z^{(0)}$ is initialized as the adjacency matrix of a regular ring lattice. In other words, it has $0$s everywhere except in the extended $K$-band, where it has $1$s. Every edge of this regular ring lattice is removed independently with probability $p_n$ during step~\ref{item:updateZ}.

$U^{(0)}$ is initialized as an empty matrix. It is updated only during step~\ref{item:updateU}. Every entry $ij$ in the upper half has two opportunities to be switched to $1$: when processing node $i$ and node $j$. Every time, it is switched to $1$ with probability
\[
    1-\sqrt{1-\frac{p_nK_n}{n}}.
\]Therefore, it has an overall probability of
\[
    1-\left(1-\left(1-\sqrt{1-\frac{p_nK_n}{n}}\right)\right)^2 = \frac{p_nK_n}{n}
\]to be switched to $1$. Finally, notice that all entries in the upper half are switched to $1$ independently of one another.\\

\noindent\textit{Proof of (\ref{item:secondsm}).} For ease of notation, we omit the index $n$ when it is clear from the context.

First, note that since entries in $U$ and $\widetilde{U}$ are in $\{0,1\}$,
\[
    \frac{1}{n}\tr\left(\frac{U-\widetilde{U}}{\sqrt{pK}}\right)^2 = \frac{1}{npK}\nnz(U-\widetilde{U}),
\]to which we can apply the inequality~\eqref{eq:inv} from the coupling construction:
\[
    \frac{1}{n}\tr\left(\frac{U-\widetilde{U}}{\sqrt{pK}}\right)^2 \le \frac{1}{npK}\sum_{i=1}^n(4|\mathcal{C}^i\cap I^i|+2||E^i\cap\mathcal{C}^i|-r_i|),
\]where we recall that $r_i$ is the number of $K/2$ rightmost edges from node $i$ that were rewired at step~\ref{item:updateZ} (switched to $0$ in $Z$), that $\mathcal{C}^i$ is the set of candidate new neighbors of $i$, defined at step~\ref{item:defineC}, that $I^i$ is the set of nodes that are neighbors of $i$ in the regular ring lattice, and that $E^i$ is the set of nodes that are not adjacent to $i$ in $\mathcal{E}$ before processing node $i$.

For all $i\in\{1,\dots,n\}$, set $u_i=|E^i\cap\mathcal{C}^i|$ and $a_i=|\mathcal{C}^i\cap I^i|$.

With those notations, the invariant identity reduces to
\[
    \nnz(U-\widetilde{U})\le\sum_{i=1}^n(4a_i+2|u_i-r_i|).
\]

We now need the following lemma.
\begin{lemma}\label{lm:didistribution} For all $i\in\{1,\dots,n\}$, $u_i$ and $r_i$ are independent, and
    \[
        \var(r_i)\le pK,\qquad|\mathbb{E}[u_i-r_i]|\le\frac{2pK^2}{n},\qquad\var(u_i)\le pK\quad\mbox{ and }\quad\mathbb{E}[a_i]\le \frac{pK^2}{n}.
    \]
\end{lemma}

We defer the proof of this lemma for now and proceed to prove claim~\ref{item:secondsm} of Proposition~\ref{prop:coupling}.
We have
\begin{align*}
    \mathbb{E}[\nnz(U-\widetilde{U})]&\le4\mathbb{E}\left[\sum_{i=1}^na_i\right]+2\sum_{i=1}^n\mathbb{E}[|u_i-r_i|]\\
    &\le4pK^2+2\sum_{i=1}^n\sqrt{\mathbb{E}[(u_i-r_i)^2]}\\
    &\le4pK^2+2\sum_{i=1}^n\sqrt{\var(u_i)+\var(r_i)+(\mathbb{E}[u_i-r_i])^2}\\
    &\le4pK^2+2\sum_{i=1}^n\sqrt{2pK+4p^2\frac{K^4}{n^2}}=4pK^2+2n\sqrt{2pK+4p^2\frac{K^4}{n^2}}\\
    &\le 8(pK^2+n\sqrt{pK}),
\end{align*}where the second inequality holds by the Jensen inequality, the third by the independence of $u_i$ and $r_i$, the fourth inequality by Lemma~\ref{lm:didistribution}, and the last one by the fact that $\sqrt{a+b}\le\sqrt{a}+\sqrt{b}$ for all $a,b\ge0$. Now,
\begin{align*}
    \mathbb{E}\left[\frac{1}{n}\tr\left(\frac{U-\widetilde{U}}{\sqrt{pK}}\right)^2\right]&=\frac{1}{npK}\mathbb{E}\left[\nnz(U-\widetilde{U})\right]\\
    &\le\frac{8(pK^2+n\sqrt{pK})}{npK}\\
    &\le 8\left(\frac{K}{n}+\frac{1}{\sqrt{pK}}\right).
\end{align*}

It remains to prove Lemma \ref{lm:didistribution}.
\begin{proof}[Proof of Lemma~\ref{lm:didistribution}] Since each of the $K/2$ edges connecting $i$ to its $K/2$ rightmost neighbors is cut off independently with probability $p$, it is independent of $u_i$ and
\[
    r_i=\binomial(K/2,p).
\]This proves $\var(r_i)\le pK$.

We now compute the distribution of $u_i$. Notice that this distribution highly depends on $i$. For ease of notation, we write $q=\frac{pK}{n}$.

Fix $i\in\{1,\dots,n\}$. In $\mathcal{E}$, before step~\ref{item:updateU} when processing node $i$, node $i$ is only adjacent to nodes of smaller indices. Therefore, for all $j>i$,
\[
    \mathbb{P}(j\in E^i\cap\mathcal{C}^i) = \mathbb{P}(j\in\mathcal{C}^i)=1-\sqrt{1-q}.
\]Additionally, for all $j\in\{1,\dots,i-1\}$, the edge $\{j,i\}$ was added with probability $1-\sqrt{1-q}$ when processing node $j$, and hence the two events $(j\in E^i)$ and $(j\in\mathcal{C}^i)$ are independent and have probability $\sqrt{1-q}$ and $1-\sqrt{1-q}$, respectively.

We obtain
\begin{equation}\label{eq:didistrib}
    u_i=\binomial((i-1),\sqrt{1-q}(1-\sqrt{1-q}))+\binomial((n-i),1-\sqrt{1-q})
\end{equation}where the sum denotes the sum of two independent random variables, each following the specified distribution. Now,
\[
    \mathbb{E}[u_i] = (i-1)\sqrt{1-q}(1-\sqrt{1-q})+(n-i)(1-\sqrt{1-q}).
\]Therefore, $\mathbb{E}[u_i]$ belongs to the segment $[\mathbb{E}[u_1],\mathbb{E}[u_n]]$. Since $\mathbb{E}[r_i] = \frac{Kp}{2}$,
\[
    |\mathbb{E}[u_i-r_i]|=\left|\mathbb{E}[u_i]-\frac{Kp}{2}\right|\le\max\left\{\left|\mathbb{E}[u_1]-\frac{Kp}{2}\right|,\left|\mathbb{E}[u_n]-\frac{Kp}{2}\right|\right\}.
\]
Now, using the two identities $\sup_{x\in[0,1]}\frac{|1-\sqrt{1-x}-\frac x2|}{x^2}\le 1$ and $\sup_{x\in[0,1]}\frac{|1-\sqrt{1-x}|}{x}\le1$ on the third line, we get
\begin{align*}
    \left|\mathbb{E}[u_1]-\frac{pK}{2}\right| &= \left|(n-1)\left(1-\sqrt{1-q}\right)-n\frac{q}{2}\right| \\
    &\le n\left|\left(1-\sqrt{1-q}\right)-\frac{q}{2}\right|+\left|1-\sqrt{1-q}\right|\\
    &\le nq^2+q\\
    &\le 2p\frac{K^2}{n}.
\end{align*}
Similarly, using the same identities on the fourth line, we have
\begin{align*}
    \left|\mathbb{E}[u_n]-\frac{Kp}{2}\right|&=\left|(n-1)\left(\sqrt{1-q}\left(1-\sqrt{1-q}\right)\right)-n\frac{q}{2}\right|\\
    &\le n\left|\sqrt{1-q}(1-\sqrt{1-q})-\frac{q}{2}\right|+\sqrt{1-q}\left(1-\sqrt{1-q}\right)\\
    &=n\left|\sqrt{1-q}-1+\frac{q}{2}\right|+\sqrt{1-q}\left(1-\sqrt{1-q}\right)\\
    &\le nq^2+q\\
    &\le2p\frac{K^2}{n}.
\end{align*}
We have shown so far that $|\mathbb{E}[u_i-r_i]|\le\frac{2pK^2}{n}$. We now show that $\var(u_i)\le pK$. First, for ease of notation, set
\[
    B^1=\bernoulli(\sqrt{1-q}(1-\sqrt{1-q}))\quad\mbox{ and }\quad B^2=\bernoulli(1-\sqrt{1-q}),
\]so that, by \eqref{eq:didistrib},
\begin{equation}\label{eq:varui}
    \var(u_i)=(i-1)\var(B^1)+(n-i)\var(B^2).
\end{equation}Notice that, using the identity $\sup_{x\in[0,1]}\frac{|1-\sqrt{1-x}|}{x}=1$ on the last line, we have
\begin{align*}
    \var(B^1) &= \sqrt{1-q}(1-\sqrt{1-q})(1-\sqrt{1-q}(1-\sqrt{1-q}))\\
    &\le(1-\sqrt{1-q})\\
    &\le q =\frac{Kp}{n},
\end{align*}and, similarly, using the same identity on the third line, we obtain
\begin{align*}
    \var(B^2) &=\sqrt{1-q}(1-\sqrt{1-q})\\
    &\le(1-\sqrt{1-q})\\
    &\le q =\frac{Kp}{n}.
\end{align*}Using those two bounds in~\eqref{eq:varui} concludes the claim.

Finally, since $a_i=|\mathcal{C}^i\cap I^i|$ and $|I^i| = K$, we have
\[
    a_i=\binomial\left(K,1-\sqrt{1-\frac{pK}{n}}\right),
\]and hence, using again the identity $\sup_{x\in[0,1]}\frac{|1-\sqrt{1-x}|}{x}=1$, we have
\[
    \mathbb{E}\left[a_i\right]=K\left(1-\sqrt{1-\frac{Kp}{n}}\right)\le\frac{pK^2}{n}.
\]
\end{proof}
This concludes the proof of Lemma~\ref{lm:didistribution}, and therefore the proof of claim~\ref{item:secondsm} of Proposition~\ref{prop:coupling}.
\end{proof}

\section{Proof of Theorem~\ref{thm:mainresult}}\label{sec:proofmainresult}
\subsection{Preliminaries} Before delving into the proof of Theorem \ref{thm:mainresult}, we need the following lemma about the adjacency matrix of a regular ring lattice of degree $K$ on $n$ vertices.  For ease of notation, let $D_{n,K}$ be the adjacency matrix of such a regular ring lattice.

\begin{lemma}\label{lm:circulantmatrix}
    For all $n\ge1$ and even integers $K<n$, the eigenvalues of $D_{n,K}$ are
    \[
        \lambda_j = \begin{cases}
            \frac{\sin((K+1)\pi \frac jn)}{\sin(\pi \frac jn)}-1&\quad\mbox{ for } j\in\{1,\dots,n-1\},\\
            K&\quad\mbox{ for }j=n.
        \end{cases}
    \]
    Furthermore, for all sequences $(K_n)_{n\ge1}$ of even integers that satisfy $2\le K_n<n$ for all sufficiently large values of $n$ and a sequence $(p_n)_{n\ge1}$ in $(0,1]$, under the assumption~\eqref{eq:assumption1}, there exist two matrices $D^{(1)}_{n,K_n}$ and $D^{(2)}_{n,K_n}$ such that
    \begin{gather}\label{eq:dndecomposition1}
        D_{n,K} = D^{(1)}_{n,K_n} + D^{(2)}_{n,K_n},\\\label{eq:dndecomposition2}
        \operatorname{rank}\left(\frac{D^{(1)}_{n,K_n}}{\sqrt{p_nK_n}}\right)\underset{n\to\infty}{=}o(n),\\\label{eq:dndecomposition3}
        \left\|\frac{D^{(2)}_{n,K_n}}{\sqrt{p_nK_n}}\right\|\toninfty0,
    \end{gather}where $\|\cdot\|$ denotes the spectral norm of a matrix. It follows that the empirical spectral measure $\mu_{\frac{D_{n,K_n}}{\sqrt{p_nK_n}}}$ of $\frac{D_{n,K_n}}{\sqrt{p_nK_n}}$ converges weakly to the Dirac mass $\delta_0$ as $n \to \infty$. 
\end{lemma}
\begin{proof}
    Let $n\ge1$ and $K<n$ denote an even integer. Notice that $D_{n,K}$ is the circulant matrix
    \[
      D_{n,K}=\begin{pmatrix}
          c_0 & c_{n-1} & \dots & c_2 & c_1\\
          c_1 & c_0 & c_{n-1} & &  c_2\\
          \vdots & c_1 & c_0 & \ddots & \vdots\\
          c_{n-2} &  & \ddots & \ddots & c_{n-1}\\
          c_{n-1} & c_{n-2} & \dots & c_1 & c_0
      \end{pmatrix}  
    \]with coefficients $c_1=\dots=c_{K/2}=c_{n-K/2}=\dots=c_{n-1}=1$ and $c_0=c_{K/2+1}=\dots=c_{n-K/2-1}=0$. It is well known (see for instance \cite{Gray2006Toeplitz}, or \cite{Davis1979} for a comprehensive study of circulant matrices) that its eigenvalues are, for all $j\in\{1,\dots,n\}$,
    \[
        \lambda_j = \sum_{k=0}^{n-1}c_ke^{2i\pi kj/n}.
    \]In our setting, we obtain
    \[
        \lambda_n = K,
    \]and, for $j\in\{1,\dots,n-1\}$,
    \begin{align*}
        \lambda_j &= \sum_{k=-K/2}^{K/2}e^{2i\pi kj/n}-1 \\
        &=e^{-2i\pi\frac{Kj}{2n}}\frac{1-e^{2i\pi(K+1)j/n}}{1-e^{2i\pi j/n}}-1\\
        &=\frac{\sin((K+1)\pi \frac jn)}{\sin(\pi \frac jn)}-1.
    \end{align*}
    This proves the first part of the lemma.

    Let $(K_n)_{n\ge1}$ and $(p_n)_{n\ge1}$ be any two sequences that satisfy the assumptions of the lemma. We will show that there are $n-o(n)$ eigenvalues that are uniformly small compared to $\sqrt{p_nK_n}$; these will be used to form the $D_{n, K_n}^{(2)}$ matrix. The remaining $o(n)$ eigenvalues will be used to form the $D_{n, K_n}^{(1)}$ matrix.  

    For ease of notation, we write $D_n$ for $D_{n,K_n}$. We introduce the orthogonal matrix $P_n$ that diagonalizes $D_n$:
    \[
        D_n = P_n\diag(\lambda_1,\dots,\lambda_n)P_n^{-1}.
    \]We set
    \[
        D^{(1)}_n = P_n\diag(\lambda_1,\dots,\lambda_{k_1},0,\dots,0,\lambda_{k_2},\dots,\lambda_n)P_n^{-1},
    \]where
    \[
        k_1 = \left\lfloor\frac{n}{\sqrt[3]{p_nK_n}}\right\rfloor\quad\mbox{ and }\quad k_2=\left\lceil n-\frac{n}{\sqrt[3]{p_nK_n}}\right\rceil.
    \]We also set
    \[
        D^{(2)}_n = D_n-D^{(1)}_n.
    \]On the one hand,
    \[
        \operatorname{rank}\left(\frac{D^{(1)}_n}{\sqrt{p_nK_n}}\right)\le\left\lfloor\frac{n}{\sqrt[3]{p_nK_n}}\right\rfloor+n-\left\lceil n-\frac{n}{\sqrt[3]{p_nK_n}}\right\rceil+1\underset{n\to\infty}{=}o(n).
    \]On the other hand, for all $j\in\left\{\left\lfloor\frac{n}{\sqrt[3]{p_nK_n}}\right\rfloor+1,\dots,\left\lceil n-\frac{n}{\sqrt[3]{p_nK_n}}\right\rceil-1\right\}$ and for $p_nK_n$ sufficiently large,
    \[
        |\lambda_j|\le\left|\frac{\sin((K_n+1)\pi \frac jn)}{\sin(\pi \frac jn)}\right|+1\le\frac{1}{\sin\left(\frac{\pi}{\sqrt[3]{p_nK_n}}\right)}+1\le C\sqrt[3]{p_nK_n}+1,
    \]where $C=\sup_{x\ge2}\frac{1}{\sqrt[3]{x}\sin\left(\frac{\pi}{\sqrt[3]{x}}\right)}=\frac{1}{\sqrt[3]{2}\sin\left(\frac{\pi}{\sqrt[3]{2}}\right)}$, and therefore
    \begin{align*}
        \left\|\frac{D_n^{(2)}}{\sqrt{p_nK_n}}\right\|&=\frac{1}{\sqrt{p_nK_n}}\sup\left\{|\lambda_j|:j\in\left\{\left\lfloor\frac{n}{\sqrt[3]{p_nK_n}}\right\rfloor+1,\dots,\left\lceil n-\frac{n}{\sqrt[3]{p_nK_n}}\right\rceil-1\right\}\right\}\\
        &\le C(p_nK_n)^{-\frac16}+(p_nK_n)^{-\frac12}\toninfty0.
    \end{align*}This shows claims \eqref{eq:dndecomposition1}-\eqref{eq:dndecomposition3} of the lemma.

    Now, by Theorems A.43 and A.45, along with Remark A.40, from \cite{bai_silverstein_2010} (recalled as Theorem~\ref{thm:thm43} and Theorem~\ref{thm:thm45} in Appendix~\ref{sec:tools}), we have
    \[
        \mu_{\frac{D_n}{\sqrt{p_nK_n}}}\toninfty\delta_0\qquad\mbox{ weakly}.
    \]
\end{proof}

\subsection{Core of the proof} Let $(p_n)_{n\ge1}$ and $(K_n)_{n\ge1}$ be two sequences as in Theorem~\ref{thm:mainresult}, satisfying in particular \eqref{eq:assumption1} and \eqref{eq:assumption2}. We also introduce the two matrices $Z_n$ and $U_n$ defined in Theorem~\ref{thm:coupling} which satisfy \eqref{eq:diffconv}. The following lemmas progressively replace $Z_n$ and $U_n$ while preserving the asymptotic behavior of the empirical eigenvalue distribution; their proofs are based on Theorem \ref{thm:replacement} from Appendix \ref{sec:replacement}. This reduces the problem to classes of matrices that have been extensively studied in the literature, allowing us to derive the limiting spectral distribution. The proof is postponed until the end of this subsection, after the statements of the lemmas.

For simplicity, we introduce the following notation.  For a random matrix $M$, we let $M^\circ = M - \mathbb{E}[M]$, where the expectation acts entrywise on the matrix elements.

We emphasize that, throughout the following lemmas, the sequences $(p_n)_{n\ge1}$ and $(K_n)_{n\ge1}$, satisfy \eqref{eq:assumption1} and \eqref{eq:assumption2}.

\begin{lemma}\label{lm:centerUZ}
We have
    \[
        \mu_{\frac{{Z}_n+{U}_n}{\sqrt{p_nK_n}}}-\mu_{\frac{{Z}_n^\circ+{U}_n^\circ}{\sqrt{p_nK_n}}}\toninfty0\qquad\mbox{ weakly in probability.}
    \]
\end{lemma}
\begin{proof}
    First, $\mathbb{E}\left[\frac{U_n}{\sqrt{p_nK_n}}\right] = \frac{\sqrt{p_nK_n}}{n}(J_n-I_n)$ where $J_n$ is the matrix with all entries equal to $1$ and $I_n$ is the identity matrix.  Notice that
    \[
        \operatorname{rank}\frac{\sqrt{p_nK_n}}{n}J_n = 1 \underset{n\to\infty}{=}o(n)
    \]and
    \[
        \left\|\frac{\sqrt{p_nK_n}}{n}I_n\right\|=\frac{\sqrt{p_nK_n}}{n}\toninfty0.
    \]Now, by the definition of $Z_n$, $\mathbb{E}[Z_n] = (1-p_n)D_n$, where $D_n$ is the adjacency matrix of a regular ring lattice of degree $K$ on $n$ vertices. Therefore, using the decomposition $D_n=D^{(1)}_n+D^{(2)}_n$ from Lemma~\ref{lm:circulantmatrix},
    \[
        \operatorname{rank}\left(\frac{\sqrt{p_nK_n}}{n}J_n+(1-p_n)\frac{D^{(1)}_n}{\sqrt{p_nK_n}}\right)\underset{n\to\infty}{=}o(n)
    \]and
    \[
        \left\|-\frac{\sqrt{p_nK_n}}{n}I_n+(1-p_n)\frac{D^{(2)}_n}{\sqrt{p_nK_n}}\right\|\toninfty0.
    \]Finally, since
    \[
        \mathbb{E}\left[\frac{Z_n+U_n}{\sqrt{p_nK_n}}\right] =\left(\frac{\sqrt{p_nK_n}}{n}J_n+(1-p_n)\frac{D^{(1)}_n}{\sqrt{p_nK_n}}\right)+\left(-\frac{\sqrt{p_nK_n}}{n}I_n+(1-p_n)\frac{D^{(2)}_n}{\sqrt{p_nK_n}}\right)
    \]then, by Theorem~\ref{thm:thm43} and Theorem~\ref{thm:thm45}, we obtain the desired conclusion.  
\end{proof}

The Gaussian Orthogonal Ensemble (GOE) of size $n$ is the distribution on real symmetric $n \times n$ matrices of the form $G_n = \frac{1}{\sqrt{n}}X_n$, where $X_n$ is an $n\times n$ symmetric random matrix whose diagonal and upper-triangular entries are independent and, for $i \leq j$, the $(i,j)$-entry has normal distribution $\mathcal{N}(0,1+\delta_{ij})$.
Here $\delta_{ij}$ is the Kronecker delta.  
For the remainder of the proof, we let $G_n$ be an $n \times n$ matrix drawn from the GOE, independent of all other sources of randomness.   
\begin{lemma}\label{lm:replaceU}
    We have
    \[
        \mu_{\frac{{Z}_n^\circ+{U}_n^\circ}{\sqrt{p_nK_n}}}-\mu_{\frac{{Z}_n^\circ}{\sqrt{p_nK_n}} + G_n} \toninfty 0\qquad\mbox{ weakly in probability}. 
    \]
\end{lemma}
\begin{proof}
    We will show 
    \begin{equation} \label{eq:showUnconv}
        \mu_{\frac{{Z}_n^\circ+{U}_n^\circ}{\sqrt{p_nK_n}}}-\mu_{\frac{{Z}_n^\circ}{\sqrt{p_nK_n}} + \sqrt{n-{p_nK_n}}G_n} \toninfty 0\qquad\mbox{ weakly in probability}. 
    \end{equation}
    Theorem \ref{thm:cor41} and the law of large numbers (along with the assumption that $\frac{p_nK_n}{n} \to 0$) imply
    \[ 
        \mu_{\frac{{Z}_n^\circ}{\sqrt{p_nK_n}} + \sqrt{n-{p_nK_n}}G_n} - \mu_{\frac{{Z}_n^\circ}{\sqrt{p_nK_n}} + G_n} \toninfty 0\qquad\mbox{ weakly in probability},
    \]
    which would then complete the proof.

    To establish \eqref{eq:showUnconv}, we will apply Theorem \ref{thm:replacement} in Appendix \ref{sec:replacement}.
    We will condition on the matrix $Z_n$ and treat $Z_n^\circ$ as deterministic for the remainder of the proof.
    The deterministic matrix $B_n$ from Theorem \ref{thm:replacement} will be taken to be $\frac{Z_n^\circ}{\sqrt{p_nK_n}}$. 
    The sparsity profile $A_n$ will be the all-ones matrix. 
    The value $K_n$ in Theorem \ref{thm:replacement} will be taken to be $n$. 
    Finally, the matrices $X_n$ and $Y_n$ in the statement of Theorem \ref{thm:replacement} will be played by $\sqrt{n}\sqrt{n - p_nK_n}G_n$ and $\frac{U_n^\circ}{\sqrt{p_nK_n/n}}$.
    The convergence in \eqref{eq:showUnconv} is precisely the conclusion of Theorem \ref{thm:replacement}, and so it only remains to check that the three main assumptions of Theorem \ref{thm:replacement} are satisfied. 

    The moment matching assumption is trivial and follows from the scaling (it is the reason we scale $G_n$ by $\sqrt{n- {p_nK_n}}$). 
    The variance of the off-diagonal entries for both matrices is $1 - \frac{p_nK_n}{n} \leq 1$, and hence assumption \eqref{assump:sparsity} is satisfied. 
    Finally, since $K_n p_n \to \infty$, it follows that, for any $\eps > 0$, $\frac{|(U^\circ_n)_{ij}|}{\sqrt{p_nK_n/n}} \leq \eps \sqrt{n}$ for all $n$ sufficiently large. 
    This verifies assumption \eqref{eq:lind} (the GOE matrix $\sqrt{n - p_nK_n}G_n$ trivially satisfies this assumption since the third absolute moment is bounded by a constant).  
    The assumptions of Theorem \ref{thm:replacement} are satisfied, and \eqref{eq:showUnconv} follows. 
\end{proof}

\begin{lemma}\label{lm:pequal1}
    Under the additional assumption that $p_n\toninfty1$, we have
    \[
        \mu_{G_n+\frac{{Z}_n^\circ}{\sqrt{p_nK_n}}}\toninfty \mu_{\mathrm{SC}}\qquad\mbox{ weakly in probability}.
    \]
\end{lemma}
\begin{proof}
    If $p_n\toninfty1$, then
    \[
        \frac{1}{n}\mathbb{E}\left[\tr\left(\frac{Z_n^\circ}{\sqrt{p_nK_n}}\right)^2\right] = (1-p_n)\toninfty0.
    \]Therefore, the second spectral moment of the difference between $G_n+\frac{{Z}_n^\circ}{\sqrt{p_nK_n}}$ and $G_n$ goes to $0$. Hence by Theorem~\ref{thm:cor41} and Markov's inequality, we have:
    \[
        \mu_{G_n+\frac{{Z}_n^\circ}{\sqrt{p_nK_n}}}-\mu_{G_n}\toninfty0\qquad\mbox{ weakly in probability}.
    \]Since $\mu_{G_n}$ converges weakly to the semicircle law (see, for example, \cite{anderson2010introduction}), we have the desired result.
\end{proof}

\begin{lemma}\label{lm:gaussianZ}
    Under the additional assumption that $p_n\toninfty p\in[0,1)$, we have
    \[
        \mu_{G_n+\frac{{Z}_n^\circ}{\sqrt{p_nK_n}}}-\mu_{G_n+\sqrt{1-p}B_n}\toninfty0\qquad\mbox{ weakly in probability},
    \]where $B_n$ is an $n\times n$ symmetric random matrix independent of $G_n$ and whose upper-triangular entries are independent and verify, for all $i\le j$,
    \[
        (B_n)_{ij}=\begin{cases}
            \mathcal{N}\left(0,\frac{1}{{K_n}}\right) & \mbox{if }\quad 0 < \min\big(|i-j|,\, n - |i-j|\big) \le \frac{K_n}{2}, \\
            0 & \mbox{otherwise,}
        \end{cases}
    \]
    with $\mathcal{N}\left(0,1\right)$ the standard normal distribution. 
\end{lemma}
\begin{proof}
    By conditioning on $G_n$, we can treat it as deterministic.  
    In this case, the result then follows by applying Theorem \ref{thm:replacement} in Appendix \ref{sec:replacement}.
    Indeed, $G_n$ will become the mean term (which is called $B_n$ in Theorem \ref{thm:replacement}). 
    The sparsity profile $A_n$ will follow the band structure of $Z_n$ and $B_n$:
    \[
        (A_n)_{ij}=\begin{cases}
            1 & \mbox{if }\quad 0 < \min\big(|i-j|,\, n - |i-j|\big) \le \frac{K_n}{2}, \\
            0 & \mbox{otherwise.}
        \end{cases}
    \]
    The matrices $X_n$ and $Y_n$ from Theorem \ref{thm:replacement} will be played by $Z_n^\circ/\sqrt{p_n}$ and $\sqrt{K_n(1-p_n)}B_n$, and $K_n$ in Theorem \ref{thm:replacement} is the same as it is here.
    Assumption \eqref{assump:sparsity} can be checked trivially, and assumption \eqref{eq:lind} follows since $p_n K_n \to \infty$. 
    Theorem \ref{thm:replacement} then gives that
    \[
        \mu_{G_n+\frac{{Z}_n^\circ}{\sqrt{p_nK_n}}}-\mu_{G_n+\sqrt{1-p_n}B_n}\toninfty0\qquad\mbox{ weakly in probability}.
    \]
    Finally, Theorem \ref{thm:cor41} and the law of large numbers imply
    \[
        \mu_{G_n+\sqrt{1-p}B_n} - \mu_{G_n+\sqrt{1-p_n}B_n}\toninfty0\qquad\mbox{ weakly in probability},
    \]
    which completes the proof.
\end{proof}

\begin{lemma} \label{lem:convBn}
    Recalling the definition of $B_n$ introduced in Lemma~\ref{lm:gaussianZ}, we have
    \[
        \mu_{B_n}\toninfty\mu_{\mathrm{SC}}\qquad\mbox{ weakly in probability}.
    \]
\end{lemma}
\begin{proof}
    For all $n\ge1$, define $B^{(1)}_n$ as follows: for all $i,j\le n$,
    \[
        (B^{(1)}_n)_{ij}=\begin{cases}
            (B_n)_{ij} & \mbox{if }\quad|i-j|\le \frac{K_n}{2},\\
            0 &\mbox{otherwise}.
        \end{cases}
    \]
    Theorem 1 of \cite{MolchanovPasturKhorunzhii1992} claims in particular that
    \[
        \mu_{B^{(1)}_n}\toninfty\mu_{\mathrm{SC}}\qquad\mbox{ weakly in probability}.
    \]
    Now, set, for all $n\ge1$,
    \[
        B^{(2)}_n = B_n-B^{(1)}_n.
    \]
    Note that $B_n^{(2)}$ has at most $K_n$ rows that are not empty (only its upper right and lower left corners are not $0$). Therefore,
    \[
        \frac1n\operatorname{rank}(B_n-B_n^{(1)})\le\frac{K_n}{n}\toninfty0.
    \]
    By Theorem~\ref{thm:thm43}, we have
    \[
        \mu_{B_n}-\mu_{B^{(1)}_n}\toninfty0\qquad\mbox{ weakly in probability},
    \]which concludes the proof.
\end{proof}

\begin{lemma}\label{lm:GBlimit}
    Under the additional assumption $p_n\toninfty p\in[0,1)$, we obtain
    \[
        \mu_{\frac{G_n+\sqrt{1-p}B_n}{\sqrt{2-p}}}\toninfty\mu_{\mathrm{SC}}\qquad\mbox{ weakly in probability}.
    \]
\end{lemma}
\begin{proof}
    It is well known (see, for example, \cite{anderson2010introduction}) that $\mu_{G_n}$ converges weakly to the semicircle law and $G_n$ is invariant under orthogonal conjugations.  
    By Lemma \ref{lem:convBn}, $\mu_{B_n}$ also converges to the semicircle law. 
    Thus, the conclusion follows from classical results (see, for instance, \cite{MR1796022}). 
\end{proof}

We are now ready to prove Theorem~\ref{thm:mainresult}.
\begin{proof}[Proof of Theorem~\ref{thm:mainresult}]
    Let $(p_n)_{n\ge1}$ and $(K_n)_{n\ge1}$ be two sequences as in Theorem~\ref{thm:mainresult}, satisfying in particular \eqref{eq:assumption1} and \eqref{eq:assumption2}. For $A_n$ the adjacency matrix of $WS_n(K_n,p_n)$, we introduce the two matrices $Z_n$ and $U_n$ defined in Theorem~\ref{thm:coupling}, which satisfy \eqref{eq:diffconv}.

    Assume first that $(p_n)_{n\ge1}$ admits a limit $p\in[0,1]$. We distinguish the cases $p=1$ and $p<1$.

    If $p=1$, applying Theorem~\ref{thm:coupling} followed by Lemmas~\ref{lm:centerUZ},~\ref{lm:replaceU}, and~\ref{lm:pequal1} yields
    \[
        \mu_{\frac{A_n}{\sqrt{p_nK_n}}}\toninfty\mu_{\mathrm{SC}}\qquad\mbox{ weakly in probability}.
    \]Since $p_n\toninfty1$, Theorem~\ref{thm:cor41} and the law of large numbers further imply that
    \[
        \mu_{\frac{A_n}{\sqrt{p_n(2-p_n)K_n}}}\toninfty\mu_{\mathrm{SC}}\qquad\mbox{ weakly in probability}.
    \]

    If $p<1$, applying Theorem~\ref{thm:coupling} followed by Lemmas~\ref{lm:centerUZ},~\ref{lm:replaceU},~\ref{lm:gaussianZ}, and~\ref{lm:GBlimit} yields
    \[
        \mu_{\frac{A_n}{\sqrt{p_n(2-p)K_n}}}\toninfty\mu_{\mathrm{SC}}\qquad\mbox{ weakly in probability}.
    \]
    Again, by Theorem~\ref{thm:cor41} and the law of large numbers,
    \[
        \mu_{\frac{A_n}{\sqrt{p_n(2-p_n)K_n}}}\toninfty\mu_{\mathrm{SC}}\qquad\mbox{ weakly in probability}.
    \]

    We now show that the hypothesis that $(p_n)_{n\ge1}$ admits a limit in $[0,1]$ can be removed. Suppose, for contradiction, that the conclusion fails. Then there exist a bounded continuous positive function $\phi$ on $\mathbb{R}$, $\eps>0$ and $\alpha>0$, and a subsequence $(p_{n_k})_{k\ge1}$ of $(p_n)_{n\ge1}$ such that, for every $k\ge1$,
    \[
        \mathbb{P}\left(\left|\int_{\mathbb{R}}\phi(x)\mu_{\frac{A_{n_k}}{\sqrt{p_{n_k}(2-p_{n_k})K_{n_k}}}}(dx)-\int_{\mathbb{R}}\phi(x)\mu_{\mathrm{SC}}(dx)\right|\ge\eps\right)\ge\alpha.
    \]
    Since $[0,1]$ is compact, $(p_{n_k})_{k\ge1}$ admits a convergent subsequence with limit $p\in[0,1]$. By the first part of the proof, the corresponding subsequence of empirical spectral measures converges weakly to $\mu_{\mathrm{SC}}$, which is a contradiction. This completes the proof.
\end{proof}

\section{Discussion on the limiting moments of the eigenvalue distribution of $WS_n(K,p)$}\label{sec:limitingmoments}

In this section, we propose a conjectural description of the limiting moments of the empirical eigenvalue distribution of $WS_n(K,p)$ in the regime where $K$ and $p$ are fixed and do not depend on $n$. This conjecture is motivated by a local weak convergence result established by Alimohammadi, Isik, and Saberi \cite{Alimohammadi2025LocalLimitsSmallWorld}, and is further supported by numerical simulations.

Subsection~\ref{subsec:convres} recalls the main result of Alimohammadi, Isik, and Saberi and introduces the limiting random graph denoted $F(K,p)$. Subsection~\ref{subsec:conjlm} formulates our conjecture relating limiting moments to the expected number of closed rooted walks in $F(K,p)$. We derive the first five conjectured limiting moments in Subsection~\ref{subsec:momentsformulas}. Finally, we present in Subsection~\ref{subsec:momentssimulation} numerical simulations supporting the conjecture.

\subsection{Convergence result}\label{subsec:convres} We introduce the definition of the Watts--Strogatz random graph that Alimohammadi, Isik, and Saberi use in their paper. We thereafter define the limiting structure, called the full $K$-fuzz graph. We then recall their main result, that consists in the convergence of the Watts--Strogatz random graph to the full $K$-fuzz graph in a sense made precise below.

\begin{defi}[A variant of the Watts--Strogatz random graph]
Let $n\ge1$ be an integer, $K$ a positive even integer with $K<n$, and $p \in [0,1]$. The Watts--Strogatz random graph of size $n$ with parameters $(K,p)$, denoted $\widetilde{WS}_n(K,p)$, is constructed as follows:

\begin{enumerate}
    \item \textbf{Initial lattice.}  
    Let $V_n = \{1,2,\dots,n\}$ denote the set of vertices of $WS_n(K,p)$. Arrange the vertices on a clockwise cycle. Each vertex $i \in V_n$ is connected to its $K/2$ nearest neighbors on each side, forming a regular ring lattice of degree $K$.

    \item \textbf{Rewiring step.}  
    For $i = 1,2,\dots,n$ in increasing order, and for each edge $\{i,j\}$ such that $j$ is one of the $K/2$ rightmost neighbors of $i$ (taken in increasing order), independently with probability $p$, rewire the edge $\{i,j\}$ as follows: remove the edge $\{i,j\}$ and add the edge $\{i,k\}$, where $k$ is chosen uniformly at random from $V_n$.
\end{enumerate}
\end{defi}

\begin{rk}
    This variant definition allows self-loops and multiple edges. However, since the limiting structure, defined below, does not contain either self-loops or multiple edges, it is natural to hypothesize that $WS_n(K,p)$ and $\widetilde{WS}_n(K,p)$ have the same limiting structure.
\end{rk}

We now introduce the definition of the limiting structure of $\widetilde{WS}_n(K,p)$.

\begin{defi}[$K$-fuzz structure]We need a few definitions to be able to formally introduce the full $K$-fuzz graph model, denoted $F(K,p)$. See Figure~\ref{fig:fullkfuzz} for an illustration.
    \begin{itemize}
        \item \textbf{Full $K$-path.} A full $K$-path is the infinite rooted graph with vertex set $\mathbb{Z}$ in which each vertex is adjacent to its $K/2$ nearest neighbors on either side.\\
        \item \textbf{Reduced $K$-path.} A reduced $K$-path is a rooted full $K$-path from which one of the $K/2$ edges connecting the root to its rightmost neighbors chosen uniformly at random is removed.\\
        \item \textbf{Full $K$-fuzz and reduced $K$-fuzz.} A full $K$-fuzz is obtained by starting from a full $K$-path and, for each node, independently rewiring each of its rightmost edges with probability $p$. Each rewired edge becomes a shortcut to a new full $K$-fuzz, producing an independent recursive copy. In addition, each node receives $\poisson(\frac{pK}{2})$ incoming shortcuts, each connecting to the root of a reduced $K$-fuzz. The reduced $K$-fuzz is defined analogously but starts from a reduced $K$-path.
    \end{itemize}
\end{defi}

Figure~\ref{fig:fullkfuzz} illustrates a neighborhood of the root in a full $4$-fuzz.

\begin{figure}[htbp]
\centering
\begin{tikzpicture}[
    every node/.style={
        circle,
        fill,
        inner sep=1pt
    },
    edge/.style={black}
]

\node[fill=none] at (3.5,0.5) {Full $4$-fuzz};
\node[fill=none] at (-4.8,-1.8) {Reduced $4$-fuzz};
\node[fill=none] at (4.6,-2) {Full $4$-fuzz};

\node[fill=white] (l3) at (-3,0) {};
\node (l2) at (-2,0) {};
\node (l1) at (-1,0) {};
\node[circle, draw, minimum size=2, fill=none] (r) at (0,0) {r};
\node (r1) at (1,0) {};
\node (r2) at (2,0) {};
\node[fill=white] (r3) at (3,0) {};

\draw (l2) to[out=35, in=145] (r);
\draw (l1) to[out=35, in=145] (r1);
\draw[dotted] (r) to[out=35, in=145] (r2);

\draw (l1) to (r);
\draw (r) to (r1);
\draw (r1) to (r2);
\draw[dotted] (l2) to (l1);

\draw[shorten >=18pt] (r1) to[out=35, in=145] (r3);
\draw[shorten >=13pt] (r2) to (r3);

\draw[shorten <=18pt] (l3) to[out=35, in=145] (l1);
\draw[shorten <=13pt] (l3) to (l2);

\node (q1) at (2.5,-1.5) {};
\node (qr) at (3.5,-1.5) {};
\node[fill=white] (ql) at (1.5,-1.5) {};
\node[fill=white] (qr2) at (4.5,-1.5) {};
\node[fill=white] (ql2) at (1,-1.5) {};
\draw[dotted] (q1) to (qr);
\draw[shorten <=5pt] (ql) to (q1);
\draw[shorten >=9pt] (qr) to (qr2);
\draw[shorten <=16pt] (ql2) to[out=35, in=145] (q1);
\draw[shorten >=13pt] (q1) to[out=35, in=145] (qr2);

\node[fill=white] (q2) at (2,-2.5) {};
\node[fill=white] (q3) at (3,-2.5) {};

\draw[shorten <=5pt, loosely dashed] (q1) to (q2);
\draw[shorten <=5pt, dashed] (q1) to(q3);

\node (p1) at (-2,-1.5) {};
\node (pl) at (-3,-1.5) {};
\node[fill=white] (pl2) at (-4,-1.5) {};
\node[fill=white] (pr2) at (-0.5,-1.5) {};
\draw (pl) to (p1);
\draw[shorten <=13pt] (pl2) to[out=35, in=145] (p1);
\draw[shorten >=17pt] (p1) to[out=35, in=145] (pr2);
\draw[shorten <=10pt] (pl2) to (pl);

\node[fill=white] (p2) at (-3,-2.5) {};
\node[fill=white] (p3) at (-2,-2.5) {};
\draw[shorten <=5pt, loosely dashed] (p1) to (p2);
\draw[shorten <=5pt, loosely dashed] (p1) to (p3);

\draw[dashed] (r)--(q1);
\draw[loosely dashed] (r)--(p1);

\end{tikzpicture}
\caption{A rooted full $K$-fuzz for $K=4$. Edges belonging to the initial full and reduced $K$-paths are drawn as solid lines. Edges cut off from these $K$-paths are shown as dotted lines, while the corresponding connections to other full $K$-fuzz are dashed. Loosely dashed edges represent shortcut connections from reduced $K$-fuzz.}
\label{fig:fullkfuzz}
\end{figure}
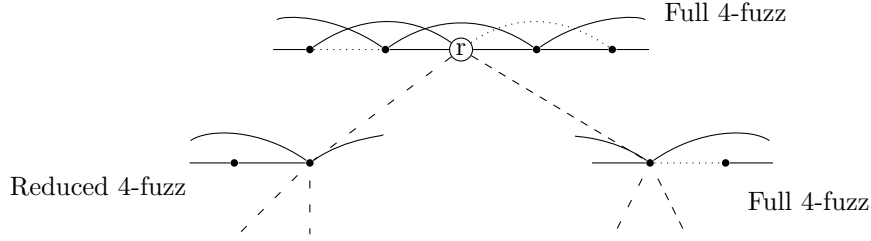

They use the theory of local convergence on the set of rooted locally finite graphs introduced by Benjamini and Schramm \cite{BenjaminiSchramm2001Recurrence} and Aldous and Steele \cite{AldousSteele2004ObjectiveMethod} to state their main theorem, which reads as follows.

\begin{thm}[{\cite[Theorem 3.1]{Alimohammadi2025LocalLimitsSmallWorld}}]We have
    \[
        \widetilde{WS}_n(K,p){\toninfty}F(K,p)\qquad\mbox{ in probability},
    \]where the convergence is the local convergence on the set of rooted locally finite graphs described in {\cite[Section 2]{Alimohammadi2025LocalLimitsSmallWorld}}.
\end{thm}

\subsection{Conjecture on the limiting moments}\label{subsec:conjlm}
To state our main conjecture, we first define what we call a closed (rooted) walk of length $k$ in a rooted graph $(G,o)$.
\begin{defi}[Closed walk]
    For any rooted graph $(G,o)$ and any $k\ge1$, a closed walk of length $k$ is a sequence of vertices $v_1,\dots,v_{k+1}$ such that $v_1=v_{k+1}$ and, for every $i\in\{1,\dots,k\}$, the edge $\{v_i,v_{i+1}\}$ is an edge of $G$. If, moreover, $v_1=o$, then the walk is called a closed rooted walk. The number of closed rooted walks of length $k$ in $(G,o)$ is denoted by $c_k(G,o)$.
\end{defi}

Our main conjecture is as follows.
\begin{conj}
    Let $K\ge1$ be an even integer and $p \in [0,1]$. Let also $A_n$ denote the adjacency matrix of $WS_n(K,p)$, for $n\in\mathbb{N}$. Then, for all $k\ge1$,
    \[
        \frac1n\tr(A_n^k)\toninfty\mathbb{E}\left[c_k(F(K,p))\right]\qquad\mbox{ in probability}.
    \]
\end{conj}

\begin{rk}This conjecture is supported by the following.
    \begin{enumerate}
        \item Since the full $K$-fuzz graph does not contain self-loops or multiple edges, it is natural to expect that such events occur with negligible probability in the construction of $\widetilde{WS}_n(K,p)$ as $n \to \infty$, and hence that $WS_n(K,p)$ and $\widetilde{WS}_n(K,p)$ have the same limiting eigenvalue distribution.

       \item For all $k \ge 1$ and all finite graphs $G$, the $k^\text{th}$ empirical moment of the adjacency matrix corresponds to the average number of closed walks of length $k$ starting at a vertex.

        \item For the limiting object $F(K,p)$, by symmetry, the expected number of closed walks of length $k$ starting at a given vertex does not depend on the chosen vertex, and therefore is equal to the number of closed rooted walks of length $k$, for all $k\ge1$.

        \item Finally, by construction, the degree of each vertex in $\widetilde{WS}_n(K,p)$ is the sum of $K/2$ and two independent Binomial random variables. As $n \to \infty$, this degree distribution converges to the sum of $K/2$, a $\binomial(K/2,1-p)$ random variable, and an independent $\poisson(pK/2)$ random variable. This suggests that degrees remain well-controlled in the large-$n$ limit. This provides additional heuristic support for the idea that the local convergence of $\widetilde{WS}_n(K,p)$ to $F(K,p)$ may extend to convergence of empirical spectral moments to the expected number of closed rooted walks in $F(K,p)$.
    \end{enumerate}

    Additionally, Subsection~\ref{subsec:momentssimulation} presents numerical simulations supporting this conjecture.
\end{rk}

\subsection{Number of closed rooted walks in $F(K,p)$.}\label{subsec:momentsformulas} In this subsection, we succinctly explain how we derive closed-form formulas for $\mathbb{E}[c_k(F(K,p))]$ for $k\le 5$. We also highlight the difficulties that arise when $k$ becomes larger.

For all $k\ge 0$, let $F(K,p,k)$ denote the subgraph of $F(K,p)$ induced by the vertices within graph distance at most $k$ from the root, and let $A_k$ denote its adjacency matrix, where the root is indexed first. We write $A_k = Z_k + U_k$, where $Z_k$ is the adjacency matrix of the subgraph of $F(K,p,k)$ containing only the edges that belong to one of the initial full $K$-paths or reduced $K$-paths used in the construction of $F(K,p)$. We denote by $\mathcal{Z}_k$ and $\mathcal{U}_k$ the two subgraphs of $F(K,p,k)$ whose adjacency matrices are $Z_k$ and $U_k$, respectively. For ease of notation, we omit the index $k$ when it is clear from the context. To compute $c_k(F(K,p))$, we use the following observations.
\begin{enumerate}
    \item Since a closed rooted walk of length $k$ only visits vertices within graph distance $k$ from the root, $c_k(F(K,p)) = (A_k^k)_{11}$.

    \item\label{rk:nosharedpaths}
    By construction, for all nodes $i$ of $F(K,p,k)$, there do not simultaneously exist paths from $i$ to the root of $F(K,p,k)$ in $\mathcal{Z}_k$ and in $\mathcal{U}_k$. In particular, for all $p,q\in\mathbb{N}$,
    \[
        (Z^pU^q)_{11} = (Z^p)_{11}\times(U^q)_{11}.
    \]

    \item \label{rk:utree}
    $\mathcal{U}_k$ is a tree, and, in particular, for all $q>0$ an odd number,
    \[
        (U^q)_{11} = 0.
    \]Additionally, for every node that is not a leaf, its offspring distribution is given by the sum of two independent random variables: a $\binomial(K/2,p)$ and a $\poisson(pK/2)$.

    \item \label{rk:dnfullkpath}
    Let $D_n$ denote the adjacency matrix of a regular ring lattice of degree $K$ on $n$ vertices. The number of closed rooted walks of length $k$ in a full $K$-path is equal to $(D_n^k)_{11}$ for $n$ sufficiently large compared to $k$. Additionally, by symmetry and since $D_n$ is deterministic, $(D_n^k)_{11} = \tau(D_n^k)$, where $\tau$ is defined in Section~\ref{sec:notation}. By Lemma~\ref{lm:circulantmatrix}, a straightforward calculation gives
    \begin{equation}\label{eq:taudn}
        \tau(D_n^k)
        = \int_0^1 \left(\frac{\sin((K+1)\pi x)}{\sin(\pi x)} - 1\right)^k dx
        = \frac{2^k}{2\pi} \int_0^{2\pi} \left(\sum_{s=1}^{K/2} \cos(sx)\right)^k dx.
    \end{equation}
\end{enumerate}

\subsubsection{First moment}
Since there are no self-loops in $F(K,p)$,
\[
    (A^1)_{11}=0.
\]

\subsubsection{Second moment}
$(A^2)_{11}$ corresponds to the degree of the root. By construction, since on average $K/2$ edges are cut off from the root in $F(K,p)$ and $K/2$ shortcuts are added,
\[
    \mathbb{E}[(A^2)_{11}]=K.
\]

\subsubsection{Third moment} We have
\[
    (A^3)_{11} = ((Z+U)^3)_{11} = (Z^3)_{11},
\]because, using observations \eqref{rk:nosharedpaths} and \eqref{rk:utree} above, and the fact that $(MN)_{11}=(NM)_{11}$ for all $n\times n$ symmetric matrices $M$ and $N$,
\[
    (U^3)_{11} = (ZU^2)_{11} = (Z^2U)_{11} = 0.
\]
Now, all closed rooted walks of length $3$ in $\mathcal{Z}$ use $3$ different edges, and therefore, by remark \eqref{rk:dnfullkpath},
\[
    \mathbb{E}[(Z^3)_{11}] = (1-p)^3\tau(D_n^3).
\]Using equation \eqref{eq:taudn}, we find $\tau(D_n^3) = \frac{3}{4}K(K-2)$. Therefore,
\[
    (A^3)_{11} = (1-p)^3\frac{3}{4}K(K-2).
\]

\begin{rk}
    It is worth mentioning that these conjectured first three moments correspond to those found by P. Nakkirt \cite{nakkirt2020eigenvalues}, which provides additional support for the conjecture.
\end{rk}

\subsubsection{Fourth moment}
Using the same technique as for the third moment, we obtain
\[
    (A^4)_{11} = (Z^4)_{11} + (U^4)_{11}+4(Z^2U^2)_{11}.
\]First, a straightforward calculation using the structure of $\mathcal{U}$ described in remark \eqref{rk:utree} above gives
\[
    \mathbb{E}[(U^4)_{11}] = 2K^2p^2+Kp(1-p).
\]Second, by counting the number of closed rooted walks of length $4$ in $\mathcal{Z}$ using one, two, three or four different edges, and by observation~\eqref{rk:dnfullkpath}, we obtain
\[
    \mathbb{E}[(Z^4)_{11}] = (1-p)K+(1-p)^22K(K-1) + (1-p)^3\times 0 + (1-p)^4(\tau(D_n^4)-2K^2+K).
\]
Finally, by observation~\eqref{rk:nosharedpaths}, $(Z^2U^2)_{11}=(Z^2)_{11}\times(U^2)_{11}$, which is equal to the product of the degree of the root in $\mathcal{Z}$ by the degree of the root in $\mathcal{U}$. Let $d$ denote the number of rightmost edges starting from the root in $F(K,p)$ that have been cut off. We have
\begin{align*}
    \mathbb{E}[(Z^2)_{11}\times(U^2)_{11}|d] &=\mathbb{E}[(Z^2)_{11}|d]\times\mathbb{E}[(U^2)_{11}|d]\\
    &=((1-p)K/2+(K/2-d))(pK/2+d)\\
    &= p(2-p)K^2/4+d\times2(1-p)K/2-d^2.
\end{align*}Since $d$ is a Binomial random variable with parameters $(K/2,p)$,
\[
    \mathbb{E}[(Z^2U^2)_{11}]=p(1-p)\left(K^2-\frac K2\right).
\]
Finally, we have
\begin{multline*}
    \mathbb{E}[(A^4)_{11}] = 2K^2p^2+Kp (1-p) + (1-p)K+(1-p)^22K(K-1) \\
    +(1-p)^4(\tau(D_n^4)-2K^2+K) + 4p(1-p)\left(K^2-\frac K2\right).
\end{multline*}

\subsubsection{Fifth moment} Using the same technique, we obtain
\[
    (A^5)_{11} = (Z^5)_{11}+5(Z^3U^2)_{11}.
\]First, for $(Z^5)_{11}$, we again count the number of closed rooted walks in a full $K$-path that use $1$ to $5$ different edges. We obtain
\[
    \mathbb{E}[(Z^5)_{11}]=(1-p)^35\tau(D_n^3)+(1-p)^43(K-2)\tau(D_n^3) + (1-p)^5\left(\tau(D_n^5) - (3K-1)\tau(D_n^3)\right).
\]
$(Z^3U^2)_{11}$ is equal to the number of closed rooted walks of length $5$ that use, in this order, $3$ edges in $\mathcal{Z}$ and $2$ in $\mathcal{U}$. By conditioning on a specific walk of length $3$ in $\mathcal{Z}$, we can deduce the average number of edges in $\mathcal{U}$ from the root. We obtain
\[
    (Z^3U^2)_{11} = (1-p)^3p\,\tau(D_n^3)(K-1).
\]Therefore, we have
\begin{multline*}
    \mathbb{E}[(A^5)_{11}]=(1-p)^35\tau(D_n^3)+(1-p)^43(K-2)\tau(D_n^3) + (1-p)^5\left(\tau(D_n^5) - (3K-1)\tau(D_n^3)\right)\\
    +5(1-p)^3p\,\tau(D_n^3)(K-1).
\end{multline*}

\subsubsection{Higher-order moments.}
For $k \ge 6$, several additional difficulties arise. First, we do not have a systematic combinatorial method to compute the expectation of $(Z^k)_{11}$. Indeed, we need to enumerate the number of closed rooted walks in $\mathcal{Z}$ of length $k$ that use only $m$ different edges, for $m\le k$, which appears to be more difficult. Second, the correlations between the number of walks in $\mathcal{Z}$ and the number of walks in $\mathcal{U}$ also become more difficult to handle. Finally, starting from $k=6$, certain terms in the expansion of $(A^k)_{11}$, such as $Z^2 U Z^2 U$, involve walks that traverse multiple full $K$-paths or reduced $K$-paths beyond the one containing the root. These terms introduce additional correlations between possible walks, and hence make the calculations harder.

In the following table, we have summarized and simplified the expressions in the formulas of the expected number of closed rooted walks of length $k$ in $F(K,p)$, for $k\in\{1,\dots,5\}$.

\begin{table}[H]
    \centering
    \renewcommand{\arraystretch}{1.6}
    \begin{tabular}{c || c}
        \hline
        $k$ & $\mathbb{E}[(A)^k_{11}]$ \\
        \hline
        1 & $0$ \\
        2 & $K$ \\
        3 & $\frac34(1-p)^3K(K-2)$ \\
        4 & $(1-p)^4\tau(D_n^4)+p(2-p)K\bigl[(2K-1)p^2 - (4K-2)p + 4K - 1\bigr]$\\
        5 & $(1-p)^5\tau(D_n^5)+\frac34p(1 - p)^3(K - 2)K\bigl[(1-3K)p + 8K - 1\bigr]$\\
        \hline
    \end{tabular}

    \caption{The expectation of the number of closed rooted walks of length $k$ in $F(K,p)$ for $k\in\{1,\dots,5\}$, and for $p\in[0,1]$ and $K\ge1$ an even integer. We conjectured that those values correspond to the first five limiting moments of the empirical eigenvalue distribution of $WS_n(K,p)$. An explicit analytic formula for $\tau(D_n^k)$ is given in~\eqref{eq:taudn}.}
    \label{tab:limitingmoments}
\end{table}

\subsection{Simulations}\label{subsec:momentssimulation}
We performed numerical simulations to illustrate and provide numerical support for the conjecture. These simulations consist of two complementary settings.

In the first setting, we estimate the third, fourth, and fifth moments of the empirical eigenvalue distribution of adjacency matrices of randomly generated Watts--Strogatz random graphs, for several values of $K$ and a range of values of $p$. These empirical moments are then compared with the corresponding expected number of closed rooted walks of the same lengths in $F(K,p)$ (see Figure~\ref{fig:momentssimulation}).

More precisely, we fix $n=3500$ and generate one realization of $WS_n(K,p)$ for each choice of $K \in {4,12,50}$ and ten values of $p$ regularly spaced in $[0,1]$. For each configuration, we compute the empirical moments for $k \in \{3,4,5\}$, and compare them with the corresponding conjectured quantities derived in Subsection~\ref{subsec:momentsformulas}. The results are displayed in Figure~\ref{fig:momentssimulation}, where we plot the empirical values together with the conjectured curves.

For $K \in \{4,12\}$, the empirical points closely follow the conjectured curves across all values of $p$. For $K=50$, a visible discrepancy appears, which is consistent with finite-size effects, as $K$ is no longer small compared to $n$. In particular, increasing $n$ significantly increases the computational cost, making this regime more difficult to explore numerically.

To further investigate this behavior, we consider a second setting, illustrated in Figure~\ref{fig:convergencepfixed}, where we study the convergence of empirical moments as $n$ increases for fixed values of $K$, $p$, and $k$. We consider two configurations: (\subref{subfig:convergengeK50}) $K=50$, $p=0.3$, $k=4$, and (\subref{subfig:convergengeK100}) $K=100$, $p=0.7$, $k=5$. For each configuration, we compute empirical moments for $n \in \{800,1600,3200,6400\}$, each averaged over four independent realizations.

In both cases, the empirical values are well approximated by a curve of the form $T+C/n$, where $T$ denotes the conjectured limiting value and $C$ is chosen by least squares fitting. The relative error between the empirical values and the fitted curve remains below $6 \times 10^{-3}$. The observed convergence towards $T$ is consistent with the conjectured limiting behavior.
\begin{figure}[htbp]
    \centering

    \begin{subfigure}{0.3\textwidth}
        \centering
        \includegraphics[width=\textwidth]{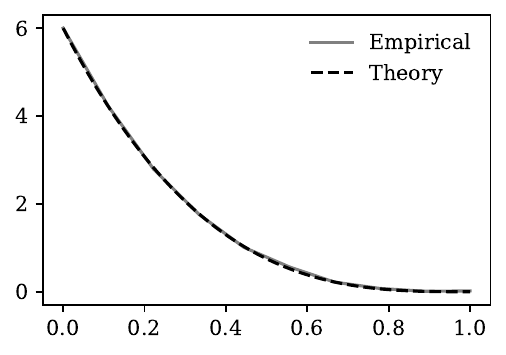}
        \caption{$K=4$, $k=3$}
    \end{subfigure}
    \hfill
    \begin{subfigure}{0.3\textwidth}
        \centering
        \includegraphics[width=\textwidth]{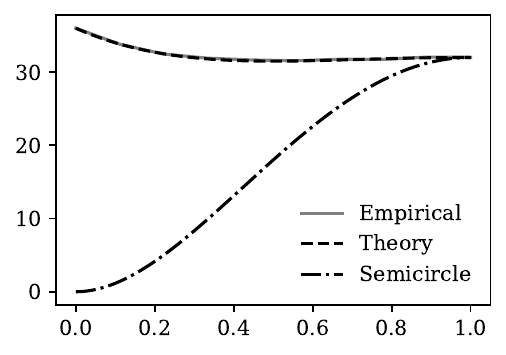}
        \caption{$K=4$, $k=4$}
    \end{subfigure}
    \hfill
    \begin{subfigure}{0.3\textwidth}
        \centering
        \includegraphics[width=\textwidth]{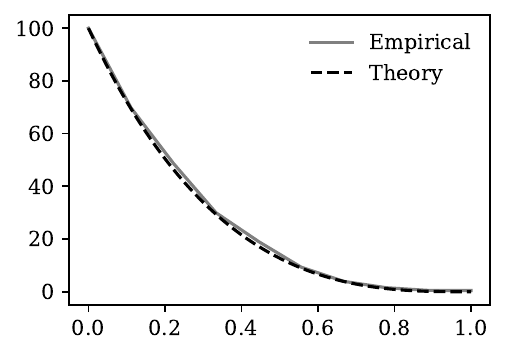}
        \caption{$K=4$, $k=5$}
    \end{subfigure}

    \vspace{0.5em}

    \begin{subfigure}{0.3\textwidth}
        \centering
        \includegraphics[width=\textwidth]{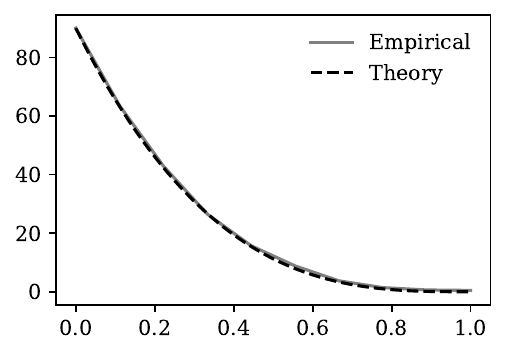}
        \caption{$K=12$, $k=3$}
    \end{subfigure}
    \hfill
    \begin{subfigure}{0.3\textwidth}
        \centering
        \includegraphics[width=\textwidth]{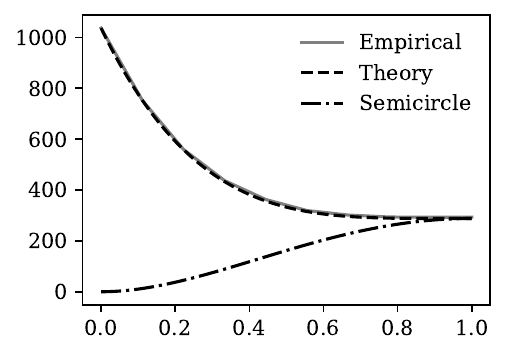}
        \caption{$K=12$, $k=4$}
    \end{subfigure}
    \hfill
    \begin{subfigure}{0.3\textwidth}
        \centering
        \includegraphics[width=\textwidth]{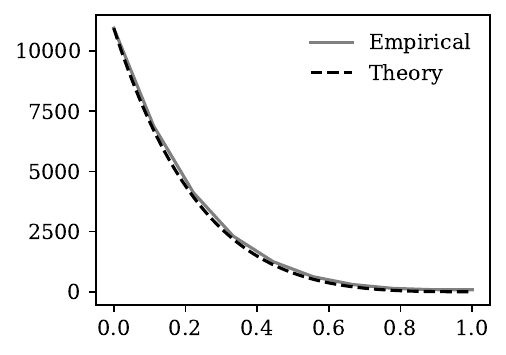}
        \caption{$K=12$, $k=5$}
    \end{subfigure}

    \vspace{0.5em}

    \begin{subfigure}{0.3\textwidth}
        \centering
        \includegraphics[width=\textwidth]{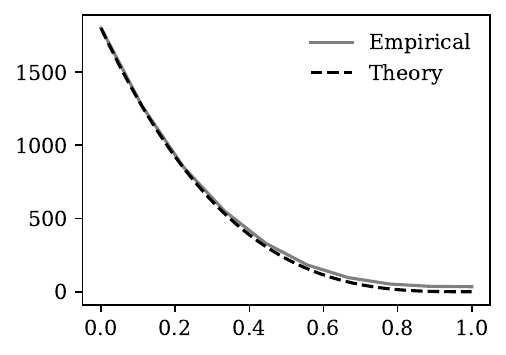}
        \caption{$K=50$, $k=3$}
    \end{subfigure}
    \hfill
    \begin{subfigure}{0.3\textwidth}
        \centering
        \includegraphics[width=\textwidth]{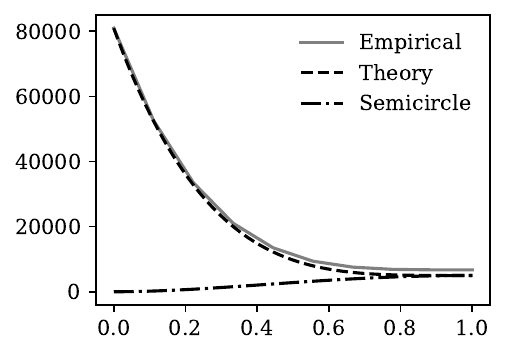}
        \caption{$K=50$, $k=4$}
    \end{subfigure}
    \hfill
    \begin{subfigure}{0.3\textwidth}
        \centering
        \includegraphics[width=\textwidth]{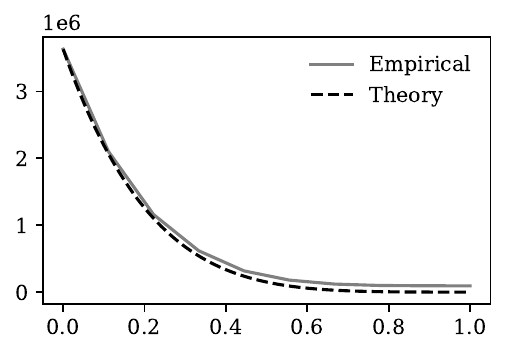}
        \caption{$K=50$, $k=5$}
    \end{subfigure}

    \caption{Empirical and theoretical limiting moments (for $k \le 5$) as functions of $p$, for different values of $K$. For $k=4$, we also plot the corresponding moments of the semicircle law, which are zero for odd orders. These moments are scaled by a factor $\left(\sqrt{Kp(2-p)}\right)^k$.}
    \label{fig:momentssimulation}
\end{figure}

\begin{figure}[htbp]
    \centering

    \begin{subfigure}{0.48\textwidth}
        \centering
        \includegraphics[width=\textwidth]{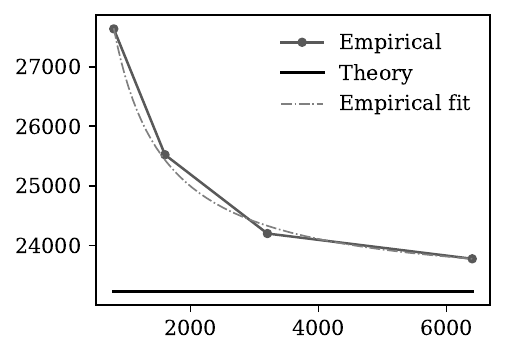}
        \caption{$K=50$, $k=4$, $p=0.3$}
        \label{subfig:convergengeK50}
    \end{subfigure}
    \hfill
    \begin{subfigure}{0.48\textwidth}
        \centering
        \includegraphics[width=\textwidth]{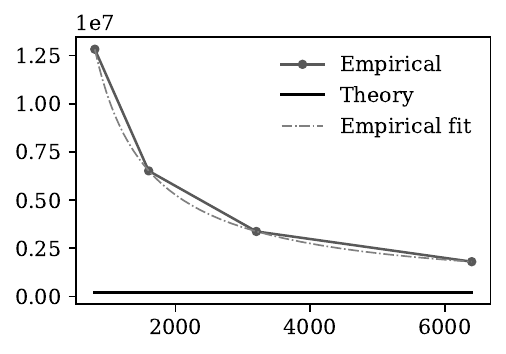}
        \caption{$K=100$, $k=5$, $p=0.7$}
        \label{subfig:convergengeK100}
    \end{subfigure}
    \caption{Empirical moments as functions of $n$, for fixed values of $K$, $p$, and $k$, illustrating convergence towards the corresponding theoretical value.}
    \label{fig:convergencepfixed}
\end{figure}
\FloatBarrier
\appendix

\section{Useful results from Bai and Silverstein}\label{sec:tools}

We recall three technical results from Bai and Silverstein \cite{bai_silverstein_2010} that are used in the proofs of our main results.  

\begin{thm}[{\cite[Theorem A.43]{bai_silverstein_2010}}]\label{thm:thm43}
    Let $(A_n)_{n\ge1}$ and $(B_n)_{n\ge1}$ be two sequences of random matrices, where for each $n$, both $A_n$ and $B_n$ are $n \times n$ Hermitian random matrices. If 
    \[
        \frac{1}{n}\operatorname{rank}(A_n-B_n)\toninfty0\qquad\mbox{ in probability},
    \]then
    \[
        \mu_{A_n}-\mu_{B_n}\toninfty0\qquad\mbox{ weakly in probability}.
    \]
\end{thm}

\begin{thm}[{\cite[Theorem A.45, Remark A.40]{bai_silverstein_2010}}]\label{thm:thm45}
    Let $(A_n)_{n\ge1}$ and $(B_n)_{n\ge1}$ be two sequences of random matrices, where for each $n$, both $A_n$ and $B_n$ are $n \times n$ Hermitian random matrices. If
    \[
        \|A_n-B_n\|\toninfty0\qquad\mbox{ in probability},
    \]then
    \[
        \mu_{A_n}-\mu_{B_n}\toninfty0\qquad\mbox{ weakly in probability}.
    \]
\end{thm}

\begin{thm}[{\cite[Corollary A.41]{bai_silverstein_2010}}]\label{thm:cor41}
    Let $(A_n)_{n\ge1}$ and $(B_n)_{n\ge1}$ be two sequences of random matrices, where for each $n$, both $A_n$ and $B_n$ are $n \times n$ Hermitian random matrices. If
    \[
        \frac{1}{n}\tr\left((A_n-B_n)^2\right)\toninfty0\qquad\mbox{ in probability},
    \]then
    \[
        \mu_{A_n}-\mu_{B_n}\toninfty0\qquad\mbox{ weakly in probability}.
    \]
\end{thm}

\section{Proof of Proposition~\ref{prop:kpto0}}\label{sec:kpto0}

This section is devoted to the proof of Proposition~\ref{prop:kpto0}.

\begin{proof}[Proof of proposition~\ref{prop:kpto0}]
    Let $(K_n)_{n\ge1}$ denote a sequence of even numbers that satisfy $0<K_n<n$, for all sufficiently large values of $n$, and let $(p_n)_{n\ge1}$ denote a sequence of numbers in $[0,1]$. We assume
    \[
        p_nK_n\toninfty0.
    \]
    Define $D_n$ as the adjacency matrix of the deterministic ring lattice with degree $K_n$. Let $A_n$ denote the adjacency matrix of $WS_n(K_n,p_n)$.
    
    Notice that
    \begin{equation}\label{eq:ranklessN}
        \operatorname{rank}(A_n-D_n)\le\nnz(A_n-D_n),
    \end{equation}where $\nnz(A_n-D_n)$ is the number of nonzero entries in $A_n-D_n$. By construction, the matrix $A_n-D_n$ has nonzero entries only where edges have been rewired. Each rewiring operation removes one edge and adds another, contributing at most 4 nonzero entries to $A_n-D_n$ (2 for the cut-off edge and 2 for the added edge, due to symmetry). Thus,
    \[
        \nnz(A_n-D_n) \leq 4 \times \mbox{(number of cut-off edges)}.
    \]Each of the $nK_n$ edges from the initial ring lattice is cut off independently with probability $p_n$. Hence, the number of cut-off edges is a Binomial random variable with parameters $(nK_n/2,p_n)$. Therefore, using \eqref{eq:ranklessN}, we get
    \[
        \mathbb{E}\left[\frac1n\operatorname{rank}(A_n-D_n)\right]\le\frac1n\mathbb{E}[\nnz(A_n-D_n)] \le 2p_nK_n\toninfty0.
    \]Theorem~\ref{thm:thm43} states that
    \[
        \mu_{\frac{A_n}{\sqrt{p_nK_n}}}-\mu_{\frac{D_n}{\sqrt{p_nK_n}}}\toninfty0\qquad\mbox{ weakly in probability.}
    \]Now, by Lemma~\ref{lm:circulantmatrix},
    \[
        \mu_{\frac{D_n}{\sqrt{p_nK_n}}}\toninfty\delta_0\qquad\mbox{ weakly in probability.}
    \]
\end{proof}

\section{Replacement principle} \label{sec:replacement}

This section is devoted to the proof of Theorem \ref{thm:replacement} below.  We begin with a few definitions. 

\begin{defi}[Inhomogeneous Wigner matrix]
    We say a random $n \times n$ real symmetric matrix $X_n$ is an \emph{inhomogeneous Wigner matrix} if the entries $\{(X_n)_{ij} : 1 \leq i \leq j \leq n\}$ are jointly independent random variables, each having finite variance and mean zero.  
\end{defi}

\begin{defi}[Sparsity profile]
    We say a deterministic $n \times n$ real symmetric matrix $A_n$ is a \emph{sparsity profile} if its entries take only the values $0$ or $1$. 
\end{defi}

In order to state our results in the generality we require, we will consider a model of the form $A_n \odot X_n + B_n$, where $A_n$ is a sparsity profile, $X_n$ is an inhomogeneous Wigner matrix and $B_n$ is any deterministic matrix.  
Here, $A_n\odot X_n$ denotes the Hadamard product of $A_n$ and $X_n$ which is the entrywise product given by 
\[ (A_n \odot X_n)_{ij} = (A_n)_{ij} (X_n)_{ij} \]
for all $1 \leq i, j \leq n$. 
Thus, the matrix $A_n$ determines which entries of $X_n$ are zero deterministically. 
The matrix $B_n$ allows the entries to have nonzero mean. 
It may seem a little strange to write things in this way.  
Indeed, $A_n \odot X_n$ is just another inhomogeneous Wigner matrix.  
However, the matrix $A_n$ allows us to naturally index which terms of the sum do not contain any randomness. 

In our main result below, we will show that if $A_n$ is a sparsity profile and $X_n$ and $Y_n$ are inhomogeneous Wigner matrices, then the asymptotic behavior of the eigenvalues of $A_n \odot X_n + B_n$ is the same as the eigenvalue behavior of $A_n \odot Y_n + B_n$, provided the entries of $X_n$ and $Y_n$ match a certain number of moments. 

\begin{defi}[Match moments] \label{def:moments}
    We say two $n \times n$ inhomogeneous Wigner matrices $X_n$ and $Y_n$ \emph{match moments} if $(X_n)_{ij}$ and $(Y_n)_{ij}$ have the same variance for each $1 \leq i < j \leq n$. 
\end{defi}

We note that we do not require any moment matching for the diagonal entries of $X_n$ and $Y_n$.
We can now state the main result of this appendix. 

\begin{thm}[Replacement principle]\label{thm:replacement}
    Let $(A_n)_{n\ge1}$, $(B_n)_{n\ge1}$, $(X_n)_{n\ge1}$, and $(Y_n)_{n\ge1}$ be sequences of $n \times n$ real symmetric matrices, where $A_n$ is a deterministic sparsity profile, $B_n$ is deterministic, and $X_n$ and $Y_n$ are both inhomogeneous Wigner matrices.  
    Let $(K_n)_{n\ge1}$ be a sequence of positive real numbers. 
    Assume:
    \begin{enumerate}
        \item (Moment matching) For each $n \geq 1$, $X_n$ and $Y_n$ match moments per Definition \ref{def:moments}.
        \item (Sparsity profile) We have
        \begin{equation} \label{assump:sparsity}
            \limsup_{n \to \infty} \frac{1}{n K_n} \sum_{1 \leq i < j \leq n} (A_n)_{ij} \left( \mathbb{E} |(X_n)_{ij}|^2 + \mathbb{E}  |(Y_n)_{ij}|^2 \right) < \infty.
        \end{equation}
        \item (Lindeberg's condition) For every $\eps > 0$
        \begin{equation} \label{eq:lind}
            \lim_{n \to \infty} \left[ \frac{1}{nK_n} \sum_{1 \leq i \leq j \leq n} (A_n)_{ij} \left( \mathbb{E} \left[ (X_n)_{ij}^2 \chi_{\{|(X_n)_{ij}| > \eps \sqrt{K_n}\} } \right] + \mathbb{E} \left[ (Y_n)_{ij}^2 \chi_{\{|(Y_n)_{ij}| > \eps \sqrt{K_n}\} } \right] \right) \right] = 0,
        \end{equation}
        where $\chi_E$ denotes the indicator function of the event $E$. 
    \end{enumerate}
    Then 
    \[ \mu_{\frac{1}{\sqrt{K_n}} A_n \odot X_n + B_n} - \mu_{\frac{1}{\sqrt{K_n}} A_n \odot Y_n + B_n} \toninfty 0 \qquad\mbox{ weakly in probability}.\]
\end{thm}

Intuitively, $K_n$ measures the sparsity of $A_n$.  For instance, if $A_n$ is a band matrix, then $A_n \odot X_n$ and $A_n \odot Y_n$ are random band matrices, and $K_n$ can be taken to be proportional to the band width.
On the other hand, if every entry of $A_n$ is one, we would take $K_n = n$ and the scaling factor $\frac{1}{\sqrt{K_n}} = \frac{1}{\sqrt{n}}$ is the standard scaling for Wigner matrices (such as the GOE). 

The proof of Theorem \ref{thm:replacement} is based on Lindeberg's replacement trick \cite{MR1544569}.  
This idea has been used a number of times before in the random matrix theory literature, see, for instance \cite{MR2294976,MR2784665,MR3403996} and references therein for a partial list. 
While there are a number of results similar to Theorem \ref{thm:replacement} in the random matrix theory literature, we could not find one that applies to the specific models we study in this paper.  

The rest of this appendix is devoted to the proof of Theorem \ref{thm:replacement}. 
We start by introducing the necessary notation. 
For any $n \times n$ real symmetric matrix $M_n$, we define the resolvent $R_{M_n}$ of $M_n$ by
\[ R_{M_n}(z) = (M_n - zI)^{-1} \]
for $z \in \mathbb{C}^+ = \{z \in \mathbb{C} : \Im(z) > 0\}$. In particular, we note the bound 
\[ \|R_{M_n}(z)\| \leq \frac{1}{\Im(z)} \]
which holds for any $z \in \mathbb{C}^+$. 
We also define the Stieltjes transform $s_{M_n}$ of $M_n$ as 
\[ s_{M_n}(z) = \frac{1}{n} \tr R_{M_n}(z), \qquad z \in \mathbb{C}^+. \]
Since $|s_{M_n}(z)| \leq \|R_{M_n}(z)\| \leq \Im(z)^{-1}$, $s_{M_n}$ is a bounded analytic function in any compact subset of $\mathbb{C}^+$. 
Given the matrix $M_n$, we let $M_n'$ be defined as
\[ M_n' = \frac{1}{\sqrt{K_n}}A_n \odot M_n + B_n. \]
In this way, the two matrices we will consider are 
\[ X_n' = \frac{1}{\sqrt{K_n}} A_n \odot X_n + B_n, \quad Y_n' = \frac{1}{\sqrt{K_n}} A_n \odot Y_n + B_n. \]

Let $A_n$, $B_n$, $X_n$, and $Y_n$ be as in the statement of Theorem \ref{thm:replacement}. 
Using standard results concerning the Stieltjes transform of probability measures (see, for example, \cite[Appendix B.2]{bai_silverstein_2010}), it suffices to show that, for any $z \in \mathbb{C}^+$, 
\[ \lim_{n \to \infty} |s_{X'_n}(z) - s_{Y'_n}(z)| = 0 \]
in probability. 
Let $1 > \eps > 0$.
We will show for any $z \in \mathbb{C}^+$ there exists a constant $C > 0$ (not depending on $\eps$) so that
\[ \lim_{n \to \infty} \mathbb{P} \left( |s_{X'_n}(z) - s_{Y'_n}(z)| > C \eps \right) = 0. \]
Since $\eps$ is arbitrary, this would complete the proof.  

We begin with a truncation argument. 
Define the matrix $\tilde{X}_n$ by
\[ (\tilde{X}_n)_{ij} = (X_n)_{ij} \chi_{\{ |(X_n)_{ij}| \leq \eps \sqrt{K_n} \}}; \]
$\tilde{Y}_n$ is defined similarly.  
Recall that
\[ \tilde{X}_n' = \frac{1}{\sqrt{K_n}} A_n \odot \tilde X_n + B_n, \quad \tilde{Y}_n' = \frac{1}{\sqrt{K_n}} A_n \odot \tilde Y_n + B_n. \]
Note that
\[ \frac{1}{n} \mathbb{E} \tr \left( (X_n' - \tilde{X}_n')^2 \right) \leq \frac{2}{nK_n} \sum_{1 \leq i \leq j \leq n} (A_n)_{ij} \mathbb{E} \left[ (X_n)_{ij}^2 \chi_{ \{ |(X_n)_{ij}| > \eps \sqrt{K_n} \} } \right] \longrightarrow 0 \]
by assumption \eqref{eq:lind}.
A similar bound holds for $\tilde{Y}_n'$. 
Thus, in view of Markov's inequality and Theorem \ref{thm:cor41}, it suffices to show, for any $z \in \mathbb{C}^+$ there exists a constant $C > 0$ (not depending on $\eps$) so that 
\begin{equation} \label{eq:stjshow}
   \lim_{n \to \infty} \mathbb{P} \left( |s_{\tilde{X}'_n}(z) - s_{\tilde{Y}'_n}(z)| > C \eps \right) = 0. 
\end{equation}
To show this, we will need the following concentration inequality.

\begin{lemma}[Concentration of the Stieltjes transform] \label{lem:concentration}
    Under the assumptions of Theorem \ref{thm:replacement}, for any $z \in \mathbb{C}^+$ , there exists $C, c > 0$ so that
    \begin{equation} \label{eq:mcdiarmid}
        \mathbb{P} \left( \left|s_{\tilde X_n'}(z) - \mathbb{E} [s_{\tilde X_n'}(z)] \right| \geq C \frac{t}{\sqrt{n}} \right) \leq C e^{-c t^2}  
    \end{equation}
    for any $t > 0$. 
    The same bound holds for $s_{\tilde Y_n'}$. 
\end{lemma}
\begin{proof}
    The bound in \eqref{eq:mcdiarmid} follows from McDiarmid’s inequality (see, for instance, \cite[Theorem 2.1.10]{MR2906465}).
    The argument is standard.  For example, one can follow the methods from \cite[Section 2.4.3]{MR2906465} with only minor modifications; we omit the details. 
\end{proof}

With Lemma \ref{lem:concentration} in hand, we return to the proof of Theorem \ref{thm:replacement}.
Indeed, we wish to establish \eqref{eq:stjshow}.
To this end, fix $z \in \mathbb{C}^+$.  By Lemma \ref{lem:concentration}, it suffices to show there exists a constant $C >0$ (not depending on $\eps$) so that
\begin{equation} \label{eq:stfinal}
    \limsup_{n \to \infty} |\mathbb{E}[s_{\tilde X_n'}(z)] - \mathbb{E}[s_{\tilde Y_n'}(z)] | \leq C \eps.
\end{equation}
We will do so by replacing each entry of $\tilde X_n$, one or two at a time, until we have constructed $\tilde Y_n$. 
At each step, we will show that changing a single entry of $\tilde X_n$ (or a pair of mirrored entries) only changes $\mathbb{E} [s_{\tilde X_n'}(z)]$ by a small amount.
    
We start by defining $E_n = \{(i, j) : 1 \leq i \leq j \leq n, (A_n)_{ij} = 1\}$ to be the upper-triangular entries of $A_n$ that are nonzero.
Fix $(a, b) \in E_n$. 
Let $Q_n$ be a random $n \times n$ real symmetric matrix, where each entry $(Q_n)_{ij}$ is either equal to $(\tilde X_n)_{ij}$ or $(\tilde Y_n)_{ij}$.
($Q_n$ is always real symmetric; so, if $i < j$ and $(Q_n)_{ij} = (\tilde X_n)_{ij}$, say, then it is assumed that $(Q_n)_{ji} = (\tilde X_n)_{ji}$.)
However, we will assume that the $(a,b)$-entry and $(b,a)$-entry are zero.  
(If $a = b$, then only one entry would be zero.)
We will write $Q_n(x)$ to denote the matrix $Q_n$ where the $(a,b)$- and $(b,a)$-entries have been replaced by $x$. 
Our goal is to establish the following lemma.

\begin{lemma} \label{lem:qtaylor}
    There exists a constant $C > 0$ so that 
    \begin{align*}
        &\left| \mathbb{E}[s_{Q_n'((\tilde{X}_n)_{ab})}(z)] - \mathbb{E}[s_{Q_n'((\tilde{Y}_n)_{ab})}(z)] \right| \\
        &\qquad\leq C \frac{(A_n)_{ab}}{\eps n K_n}   \left( \mathbb{E}[( X_n)^2_{ab}\chi_{ \{|(X_n)_{ab}| > \eps \sqrt{K_n} \} }] + \mathbb{E}[( Y_n)^2_{ab}\chi_{ \{|(Y_n)_{ab}| > \eps \sqrt{K_n} \} }] \right) \\
        &\qquad\qquad + C \frac{(A_n)_{ab}}{ n K_n} \eps \left( \mathbb{E}[|(X_n)_{ab}|^2] + \mathbb{E}[|(Y_n)_{ab}|^2] \right),
    \end{align*}
    if $a < b$,
    and 
    \[ \left| \mathbb{E}[s_{Q_n'((\tilde{X}_n)_{aa})}(z)] - \mathbb{E}[s_{Q_n'((\tilde{Y}_n)_{aa})}(z)] \right| \leq \frac{C \eps}{n}(A_n)_{aa} \]
    if $a = b$.
    The constant $C$ is independent $a,b$, and $\eps$. 
\end{lemma}
We note that Lemma \ref{lem:qtaylor} also applies to the $(a,b)$-entries when $(A_n)_{ab} = 0$ as those terms are simply zero.
We delay the proof of Lemma \ref{lem:qtaylor} for now and first complete the proof of Theorem \ref{thm:replacement}.

Applying Lemma \ref{lem:qtaylor} and the triangle inequality, we obtain 
\begin{align*}
    &\left| \mathbb{E}[s_{\tilde X_n'}(z)] - \mathbb{E}[s_{\tilde Y_n'}(z)] \right| \\
    &\leq C \sum_{1 \leq a < b \leq n} \frac{(A_n)_{ab}}{\eps n K_n}   \left( \mathbb{E}[( X_n)^2_{ab}\chi_{ \{|(X_n)_{ab}| > \eps \sqrt{K_n} \} }] + \mathbb{E}[( Y_n)^2_{ab}\chi_{ \{|(Y_n)_{ab}| > \eps \sqrt{K_n} \} } ]\right) \\
        &\qquad\qquad + C \sum_{1 \leq a < b \leq n} \frac{(A_n)_{ab}}{ n K_n} \eps \left( \mathbb{E}[|(X_n)_{ab}|^2] + \mathbb{E}[|(Y_n)_{ab}|^2] \right) \\
        &\qquad\qquad +C \sum_{a=1}^n \frac{\eps}{n}(A_n)_{aa} \\
    &= S_n^{(1)} + S_n^{(2)} + S_n^{(3)},
\end{align*}
where $S_n^{(1)}$, $S_n^{(2)}$, $S_n^{(3)}$ are the first, second, and third sum, respectively.

The third sum is bounded trivially,
\[ S_n^{(3)} \leq C \eps \]
since the entries of $A_n$ are either zero or one.
Applying assumption \eqref{assump:sparsity}, we see that
\[ \limsup_{n \to \infty} S_n^{(2)} \leq C' \eps \]
for some constant $C' > 0$ (not depending on $\eps$). 
Finally, by assumption \eqref{eq:lind}, we note that
\[ \lim_{n \to \infty} S_n^{(1)} = 0. \]
Therefore, by combining these bounds, we obtain  \eqref{eq:stfinal}, and the proof of Theorem \ref{thm:replacement} is complete.

It remains to prove Lemma \ref{lem:qtaylor}.

\begin{proof}[Proof of Lemma \ref{lem:qtaylor}]
    Throughout the proof $C$ will represent a positive constant, which may change from one line to the next line, and which does not depend on $\eps, a, b$. 
    Let us begin with the case when $a < b$. 
    Define
    \[ f(x) = s_{Q_n'(x)}(z). \]
    That is, $f(x)$ is the Stieltjes transform of 
    \[ Q_n'(x) = \frac{1}{\sqrt{K_n}} A_n \odot Q_n(x) + B_n. \]
    Recall that $z \in \mathbb{C}^+$ has been fixed.
    We will write out the Taylor expansion
    \[ f(x) = f(0) + f'(0)x + f''(x) \frac{x^2}{2} + E_{ab}, \]
    where
    $|E_{ab}| \leq C (A_n)_{ab} M |x|^3$ and
    \[ M = \sup_{x \in \mathbb{R}} |f'''(x)|. \]
    Substituting $x = (\tilde X_n)_{ab}$ into this equation and taking expectation gives
    \[ \mathbb{E} [f( (\tilde X_n)_{ab})] = \mathbb{E} [f(0)] + \mathbb{E}[(\tilde X_n)_{ab}]\mathbb{E}[f'(0)] +  \frac{1}{2} \mathbb{E}[(\tilde X_n)_{ab}^2] \mathbb{E}[f''(0)] + \mathbb{E}[E_{ab}] \]
    since $f(0)$, $f'(0)$, and $f''(0)$ are all independent of $(X_n)_{ab}$.
    Applying this to $(\tilde Y_n)_{ab}$ as well and then subtracting the two equations yields
    \begin{align} \label{eq:subbnd}
        &\left| \mathbb{E} [f( (\tilde X_n)_{ab})] - \mathbb{E} [f( (\tilde Y_n)_{ab})] \right| \\
        &\qquad\qquad\leq |\mathbb{E}[f'(0)]| \left| \mathbb{E}[(\tilde X_n)_{ab} - (\tilde Y_n)_{ab}] \right| + \frac{1}{2} |\mathbb{E}[f''(0)]| \left| \mathbb{E}[(\tilde X_n)_{ab}^2 - (\tilde Y_n)_{ab}^2] \right| + \mathcal{E}_{ab}, \nonumber
    \end{align}
    where $\mathcal{E}_{ab} \leq C (A_n)_{ab} M \mathbb{E}[|(\tilde X_n)_{ab}^3| + |(\tilde Y_n)_{ab}^3|]$. 
    To compute the derivatives $f'$, $f''$, $f'''$, we apply the resolvent identity to see that
    \[ \frac{\partial (R_{M_n}(z))_{ij}}{\partial(M_n)_{kl}} = \begin{cases}
        - R_{M_n}(z)_{ik} R_{M_n}(z)_{lj} - R_{M_n}(z)_{il} R_{M_n}(z)_{kj}, & \text{ if } k \neq l, \\
        -R_{M_n}(z)_{ik}R_{M_n}(z)_{lj}, & \text{ if } k = l.
    \end{cases} \]
    From this one can derive the bounds 
    \[ \sup_{x \in \mathbb{R}}|f'(x)| \leq C \left( \frac{(A_n)_{ab}}{n\sqrt{K_n}}\right), \quad
    \sup_{x \in \mathbb{R}}|f''(x)| \leq C \left( \frac{(A_n)_{ab}}{n{K_n}}\right), \quad
    \sup_{x \in \mathbb{R}} |f'''(x)| \leq C \left( \frac{(A_n)_{ab}}{n{K_n}^{3/2}}\right). \]
    The factors of $\sqrt{K_n}$ appear due to the normalization we have chosen and the chain rule. 
    These bounds are fairly standard to derive; see, for instance, the replacement principle argument in \cite{MR3403996}.
    We omit the details.

    Returning to \eqref{eq:subbnd}, we substitute the bounds above and arrive at
    \begin{align*}
     &\left| \mathbb{E} [f( (\tilde X_n)_{ab})] - \mathbb{E} [f( (\tilde Y_n)_{ab})] \right| \\
     &\leq  C \frac{(A_n)_{ab}}{n\sqrt{K_n}} \left( \left| \mathbb{E}[(\tilde X_n)_{ab} - (\tilde Y_n)_{ab}] \right| + \frac{\left| \mathbb{E}[(\tilde X_n)_{ab}^2 - (\tilde Y_n)_{ab}^2] \right|}{\sqrt{K_n}} + \frac{\mathbb{E}[|(\tilde X_n)_{ab}^3| + |(\tilde Y_n)_{ab}^3|]}{K_n}\right).
    \end{align*}
    Since $( X_n)_{ab}$ and $( Y_n)_{ab}$ both have mean zero, 
    \begin{align*}
    \left| \mathbb{E}[(\tilde X_n)_{ab} - (\tilde Y_n)_{ab}] \right| &= \left| \mathbb{E}[( X_n)_{ab}\chi_{ \{|(X_n)_{ab}| > \eps \sqrt{K_n} \} } - ( Y_n)_{ab}\chi_{ \{|(Y_n)_{ab}| > \eps \sqrt{K_n} \} }] \right| \\
    &\leq \frac{1}{\eps \sqrt{K_n}} \left( \mathbb{E}[( X_n)^2_{ab}\chi_{ \{|(X_n)_{ab}| > \eps \sqrt{K_n} \} }] + \mathbb{E}[( Y_n)^2_{ab}\chi_{ \{|(Y_n)_{ab}| > \eps \sqrt{K_n} \} }] \right).
    \end{align*}
    Similarly, since $( X_n)_{ab}$ and $( Y_n)_{ab}$ both have the same mean and variance,
    \begin{align*}
        \left| \mathbb{E}[(\tilde X_n)_{ab}^2 - (\tilde Y_n)_{ab}^2] \right| &\leq \mathbb{E}[( X_n)^2_{ab}\chi_{ \{|(X_n)_{ab}| > \eps \sqrt{K_n} \} }] + \mathbb{E}[( Y_n)^2_{ab}\chi_{ \{|(Y_n)_{ab}| > \eps \sqrt{K_n} \} }].
    \end{align*}
    Lastly, we have
    \begin{align*}
        \mathbb{E}[|(\tilde X_n)_{ab}^3| + |(\tilde Y_n)_{ab}^3|] &\leq \eps \sqrt{K_n} \left( \mathbb{E}[|(X_n)_{ab}|^2] + \mathbb{E}[|(Y_n)_{ab}|^2] \right).
    \end{align*}

    Combining the bounds above, we arrive at
    \begin{align*}
        &\left| \mathbb{E} [f( (\tilde X_n)_{ab})] - \mathbb{E} [f( (\tilde Y_n)_{ab})] \right| \\
        &\leq C \frac{(A_n)_{ab}}{\eps n K_n}   \left( \mathbb{E}[( X_n)^2_{ab}\chi_{ \{|(X_n)_{ab}| > \eps \sqrt{K_n} \} }] + \mathbb{E}[( Y_n)^2_{ab}\chi_{ \{|(Y_n)_{ab}| > \eps \sqrt{K_n} \} }] \right) \\
        &\qquad + C \frac{(A_n)_{ab}}{ n K_n} \eps \left( \mathbb{E}[|(X_n)_{ab}|^2] + \mathbb{E}[|(Y_n)_{ab}|^2] \right),
    \end{align*}
    which completes the proof in the case $a < b$. 

    Assume now that $a = b$.
    In this case, we expand
    \[ f(x) = f(0) + E_{aa}, \]
    where $|E_{aa}| \leq C M|x|$ and we now take
    \[ M = \sup_{x \in \mathbb{R}} |f'(x)|. \]
    Applying the same derivative bounds as above, we conclude that
    \begin{align*}
        \left| \mathbb{E} [f( (\tilde X_n)_{aa})] - \mathbb{E} [f( (\tilde Y_n)_{aa})] \right| &\leq C\frac{(A_n)_{aa}}{n \sqrt{K_n}} \left( \mathbb{E}|(\tilde X_n)_{aa}| + \mathbb{E}|(\tilde Y_n)_{aa}| \right) \\
        &\leq C \eps \frac{(A_n)_{aa}}{n},
    \end{align*}
    and the proof is complete.
\end{proof}

\bibliographystyle{abbrv}
\bibliography{refs}

\end{document}